\documentclass[11pt]{amsart}

\usepackage{hyperref}
\hypersetup{bookmarksdepth=3}
  
\usepackage{geometry}
\usepackage{mathtools}
\usepackage{amsmath,amssymb}
\usepackage{
mathrsfs}
\usepackage{enumerate}
\usepackage{esint}

\usepackage{tensor}

\usepackage{tikz}
\usepackage{forest}

\usepackage{amssymb,amsmath,amsthm}

\newtheorem{theorem}{Theorem}
\newtheorem{lemma}[theorem]{Lemma}

\newtheorem{definition}[theorem]{Definition}

\theoremstyle{remark}\newtheorem{remark}[theorem]{Remark}

\newtheorem{proposition}[theorem]{Proposition}

\usepackage{graphicx}

\newcommand{\R}{{\mathbb R}}

\newcommand{\Z}{{\mathbb Z}}

\newcommand{\C}{{\mathbb C}}

\newcommand{\T}{{\mathbb T}}

\begin{document}

{\let\thefootnote\relax\footnote{Date: March 12th,  2019. 

\textcopyright 2019 by the authors. Faithful reproduction of this article, in its entirety, by any means is permitted for noncommercial purposes.}}

\title{Knocking out teeth in one-dimensional periodic NLS. }

\author{L. Chaichenets}
\address{leonid chaichenets, department of mathematics, institute for analysis, karlsruhe institute of technology, 76128 karlsruhe, germany }
\email{leonid.chaichenets@kit.edu}

\author{D. Hundertmark}
\address{Dirk Hundertmark, department of mathematics, institute for analysis, karlsruhe institute of technology, 76128 karlsruhe, Germany, and Department of Mathematics
Altgeld Hall,
1409 W.\ Green Street
Urbana, IL 61801, USA }
\email{dirk.hundertmark@kit.edu}

\author{P. Kunstmann}
\address{peer kunstmann, department of mathematics, institute for analysis, karlsruhe institute of technology, 76128 karlsruhe, germany }
\email{peer.kunstmann@kit.edu}

\author{N. Pattakos}
\address{nikolaos pattakos, department of mathematics, institute for analysis, karlsruhe institute of technology, 76128 karlsruhe, germany }
\email{nikolaos.pattakos@kit.edu}

\begin{abstract}
{We show the existence of weak solutions in the extended sense of the Cauchy problem for the cubic nonlinear Schr\"odinger equation in one dimension with initial data $u_{0}$ in $H^{s_{1}}(\R)+H^{s_{2}}(\T), 0\leq s_{1}\leq s_{2}.$ In addition, we show that if $u_{0}\in H^{s}(\R)+H^{\frac12+\epsilon}(\T)$ where $\epsilon>0$ and $\frac16\leq s\leq\frac12$ the solution is unique in $H^{s}(\R)+H^{\frac12+\epsilon}(\T).$ Our main tool is a normal form type reduction via the use of the differentiation by parts technique.}
\end{abstract}

\maketitle
\pagestyle {myheadings}

\begin{section}{introduction and main results}
We are interested in the equation
\begin{equation}
\label{maineq}
\begin{cases} iu_{t}-u_{xx}\pm|u|^{2}u=0 &,\ (t,x)\in\mathbb R^{2}\\
u(0,x)=u_{0}(x) &,\ x\in\mathbb R\\
\end{cases}
\end{equation}
with initial data $u_{0}\in H^{s}(\R)+H^{s}(\T)$ for $s\geq0,$ where $\T\coloneqq\R/\Z$ is the one-dimensional torus, that is, the circle. The Sobolev $H^{s}$ spaces are defined as
\begin{equation}
\label{spacon}
H^{s}(\R)\coloneqq\{f\in L^{2}(\R)\ /\ \|f\|_{H^{s}(\R)}\coloneqq\Big(\int_{\R}(1+|\xi|^{2})^{s}|\hat{f}(\xi)|^{2}d\xi\Big)^{\frac12}<\infty\}
\end{equation}
and 
\begin{equation}
\label{spaper}
H^{s}(\T)\coloneqq\{f\in L^{2}(\T)\ /\ \|f\|_{H^{s}(\T)}\coloneqq\Big(\sum_{n\in\Z}(1+|n|^{2})^{s}|f_{n}|^{2}\Big)^{\frac12}<\infty\},
\end{equation}
and we will use $\langle k\rangle\coloneqq(1+|k|^{2})^{\frac{1}{2}}$ for the so--called Japanese bracket. $S(\R)$ is the Schwartz class,  $D(\T)=C^{\infty}(\T)$,  
$S'(\R)$ the tempered distributions, and  $D'(\T)$ the distributions on the torus $\T$. The Fourier transform of a function $f:\R\rightarrow\C$ is given by 
\begin{equation}
\label{four}
\hat{f}(\xi)(=\mathcal F(f)(\xi))\coloneqq\int_{\R}e^{-2\pi i\xi x}f(x)dx,\ \xi\in\R,
\end{equation}
and the Fourier coefficients of a periodic function $f:\T\rightarrow\C$ are
\begin{equation}
\label{fourper}
f_{n}\coloneqq\int_{0}^{1}e^{-2\pi i nx}f(x)dx,\ n\in\Z.
\end{equation}
In \cite{Tsu} it was proved that NLS \eqref{maineq} is locally wellposed in $L^{2}(\R)$ with guaranteed time of 
existence depending only on the $L^{2}(\R)$ norm of the initial data and since this is a conserved quantity, $\|u(t,\cdot)\|_{L^{2}(\R)}=\|u_{0}\|_{L^{2}(\R)}$ for all $t\in\R,$ it follows that the NLS \eqref{maineq} is globally wellposed in $L^{2}(\R)$. 
In \cite{Bou} it was proved that the NLS \eqref{maineq} is locally wellposed in $L^{2}(\T)$ and again by the $L^{2}(\T)$ conservation law it follows that it is globally wellposed in $L^{2}(\T)$. In \cite{GKO} the NLS \eqref{maineq} was studied for initial conditions $u_{0}\in H^{s}(\T)$ 
and in \cite{KOYY} and \cite{HY} for $u_{0}\in H^{s}(\R), s\geq0$.  
In both papers unconditional well--posedness was proved for $s\geq\frac16,$ that is uniqueness of solutions in $C([0,T],H^{s}(\T))$ (and $C([0,T],H^{s}(\R))$ respectively) without intersecting with any auxiliary function space (see \cite{TK} where this notion first appeared). 
They used a normal form reduction via the differentiation by parts technique which was originally introduced in \cite{BIT} in the study of the KdV equation for periodic initial data. 
We also refer to \cite{NP} where the last author introduced a different approach to the normal form reduction for the NLS \eqref{maineq} on $\R$ which follows closely what is done in the periodic case and is well suited for modulation spaces defined in equation \eqref{def}. 
See also \cite{CHKP} for the case of general modulation spaces. 
For textbook accounts on these type of results we refer to 
\cite{LP, Tao}, to \cite{SuSu} for a slightly more applied point of view, and, in particular, \cite{ET} for a nice discussion of the differentiation by parts technique.

Here we make the differentiation by parts approach work in a \emph{hybrid case}, namely the case where the initial data $u_{0}$ is the sum of a periodic function $w_{0}$ on $\R$ and an $L^{2}(\R)$ function $v_{0}$. A tooth, as referred to in the title of this paper, is, for example, $w_0$ restricted to one period. We think of the addition of $v_0$ to $w_0$ as eliminating, or knocking out, finitely many of these teeth in the underlying periodic signal. 

Our work is motivated by high--speed optical fiber 
communications, where in a certain approximation the behavior of 
pulses in glass--fiber cables is described by a NLS.  
A periodic signal is the simplest type of a non-decaying signal, encoding, for example, an infinite string of ones 
if there is exactly one tooth per period. However, such 
a purely periodic signal carries no information. One would like to be able to change it, at least locally.   
This leads necessarily to a hybrid formulation of the NLS where the 
signal is the sum of a periodic and a localized part. 
The localized part being able to knock out, i.e., remove, one 
or more of the teeth in the underlying periodic signal. This way one can model, 
for example, a signal consisting of two infinite blocks of ones which are separated by a single zero, or even far more complicated patterns. 
An interesting question then naturally arises:  Can the missing teeth regrow, which means that the original signal gets distorted (in optics this phenomenon is known as \emph{ghost pulses}, see e.g. \cite{MM} or \cite{ZM}).  Is there an optimal choice of a periodic signal, which makes this distortion very weak or even  impossible?

From a mathematics point of view, in order to be able to address 
these type of questions, one should have first solved the corresponding 
local existence and uniqueness problems, which is 
the main purpose of this work: We solve 
the local existence problem and provide an unconditional 
uniqueness result.  Since the underlying periodic signal can also 
be the constant function, we also cover the case of so--called 
dark solitons, that is, NLS with a non-zero boundary conditions 
at infinity, where the signals are of the form $u= c+v$ with $c$ 
a constant, see \cite{KFC-G} and \cite{KL-D} for a review on 
dark solitons from 
a point of view of applied mathematics and physics. 

Our solution of NLS \eqref{maineq} with initial data $u_{0}=v_{0}+w_{0}\in H^{s_{1}}(\R)+H^{s_{2}}(\T)$ will be constructed as the sum of the solutions of the following partial differential equations
\begin{equation}
\label{perpe}
\begin{cases} iw_{t}-w_{xx}\pm|w|^{2}w=0 &,\ (t,x)\in\R\times\T\\
w(0,x)=w_{0}(x)\in H^{s_{2}}(\T) &,\ x\in\mathbb T \, ,\\
\end{cases}
\end{equation}
which is the periodic cubic NLS on the real line, and the modified cubic NLS
\begin{equation}
\label{nonperpe}
\begin{cases} iv_{t}-v_{xx}\pm G(w,v)=0 &,\ (t,x)\in\R\times\R\\
v(0,x)=v_{0}(x)\in H^{s_{1}}(\R) &,\ x\in\mathbb R\, ,\\
\end{cases}
\end{equation}
where $G(w,v)$ is the nonlinearity
\begin{equation}
\label{nonli}
G(w,v)=|w+v|^{2}(w+v)-|w|^{2}w=|v|^{2}v+v^{2}\bar{w}+w^{2}\bar{v}+2w|v|^{2}+2v|w|^{2}.
\end{equation}

In order to give a meaning to solutions of NLS \eqref{perpe} in $C([0,T], H^{s}(\T))$ and NLS \eqref{nonperpe} in $C([0,T], H^{\tilde{s}}(\R)), s, \tilde{s}\in\R$ 
and to the nonlinearities $\mathcal N(w)\coloneqq w|w|^{2}$ and $G(w,v)$ we need the following definitions, 
which first appeared in \cite{C1} for the periodic NLS. 

\begin{definition}
\label{def1}
A sequence of Fourier cutoff operators is a sequence of Fourier multiplier operators $\{T_{N}\}_{N\in \mathbb N}$ with multipliers $m_{N}:\R\to\C$ such that
\begin{itemize}
\item $m_{N}$ has compact support on $\R$ for every $N\in \mathbb N$,
\item $m_{N}$ is uniformly bounded,
\item $\lim_{N\to\infty}m_{N}(x)=1$, for any $x\in\R$. 
\end{itemize}
\end{definition}

\begin{definition}[Periodic case]
\label{def5}
Let $w\in C([0,T], H^{s}(\T))$. We say that $\mathcal N(w)$ exists and is equal to a distribution $\tilde{w}\in\mathcal [C^{\infty}((0,T),D(\T))]'$ if, for every sequence $\{T_{N}\}_{N\in\mathbb N}$ of Fourier cutoff operators, we have
\begin{equation}
\label{wknows5}
\lim_{N\to\infty}\mathcal N(T_{N}w)=\tilde{w},
\end{equation}
in the sense of distributions on $(0,T)\times\T$.
\end{definition}

\begin{definition}[Periodic case]
\label{def6} Let $r\ge 0$. 
We say that $w\in C([0,T], H^{r}(\T))$ is a weak solution in the extended sense of the NLS \eqref{perpe} if 
\begin{itemize}
\item $w(0,x)=w_{0}(x)$,
\item the nonlinearity $\mathcal N(w)$ exists in the sense of Definition \ref{def5},
\item $w$ satisfies \eqref{perpe} in the sense of distributions on $(0,T)\times\T$, where the nonlinearity $\mathcal N(w)=w|w|^{2}$ is interpreted as above.
\end{itemize}
\end{definition}
For a fixed such solution $w$ of equation  \eqref{perpe}, in the sense of Definition \ref{def6},  
we define a solution $v$ of equation \eqref{nonperpe} as 

\begin{definition}[Continuous case]
\label{def2} Let $s\ge 0$ and $v\in C([0,T], H^{{s}}(\R))$. 
 We say that $G(w,v)$ exists and is equal to a distribution $\tilde{v}\in\mathcal [C^{\infty}((0,T),S(\R)]'$ if, for every sequence $\{T_{N}\}_{N\in\mathbb N}$ of Fourier cutoff operators, we have
\begin{equation}
\label{wknows}
\lim_{N\to\infty}G(T_{N}w,T_{N}v)=\tilde{v},
\end{equation}
in the sense of distributions on $(0,T)\times\R$.
\end{definition}
Similarly to the periodic case, we also introduce
\begin{definition}[Continuous case]
\label{def3}
We say that $v\in C([0,T], H^{s}(\R))$ is a weak solution in the extended sense of NLS \eqref{nonperpe} if 
\begin{itemize}
\item $v(0,x)=v_{0}(x)$,
\item the nonlinearity $G(w,v)$ exists in the sense of Definition \ref{def2},
\item $v$ satisfies \eqref{nonperpe} in the sense of distributions on $(0,T)\times\R$, where the nonlinearity $G(w,v)$ is interpreted as above.
\end{itemize}
\end{definition}
The main results of the paper are the following

\begin{theorem}[Local existence and well--posedness]
\label{th1}
Let $0\leq s_{1}\leq s_{2}$ and $u_{0}=v_{0}+w_{0}\in H^{s_{1}}(\R)+H^{s_{2}}(\T)$. There exists a weak solution in the extended sense $u=v+w\in C([0,T], H^{s_{1}}(\R))+C([0,T],H^{s_{2}}(\T))$ of NLS \eqref{maineq} with initial condition $u_{0}$ where $w$ solves NLS \eqref{perpe} in the sense of Definition \ref{def6}, $v$ solves NLS \eqref{nonperpe} in the sense of Definition \ref{def3} and the time $T$ of existence depends only on $\|v_{0}\|_{H^{s_{1}}(\R)}, \|w_{0}\|_{H^{s_{2}}(\T)}$. 

Moreover, the solution map is locally Lipschitz continuous. 
\end{theorem}

\begin{theorem}[Unconditional uniqueness]
\label{finth2}
Let $\epsilon>0$ and $\frac16\leq s\leq\frac12$. For any initial condition $u_{0}\in H^{s}(\R)+H^{\frac12+\epsilon}(\T)$ the solution $u=v+w$ constructed in Theorem \ref{th1} is unique in $C([0,T],H^{s}(\R)+H^{\frac12+\epsilon}(\T))$.
\end{theorem}

\begin{remark}
\label{ante}
The result of Theorem \ref{finth2} is also true for $s>\frac12$, but in this case the spaces $H^{s}(\R)$ and $H^{\frac12+\epsilon}(\T)$ embed continuously into $L^\infty(\R),$ thus also their sum. Hence $H^{s}(\R) + H^{\frac12+\epsilon}(\T)$ is a Banach algebra and existence and uniqueness results become much easier with the help of straightforward direct Banach contraction mapping  arguments.
The condition $s\geq\frac16$ guarantees that $v\in H^{s}(\R)\hookrightarrow L^{3}(\R)$ which means that $|v|^{2}v\in L^{1}(\R)$ and together with $ H^{\frac12+\epsilon}(\T)\hookrightarrow L^{\infty}(\T)$, 
allows us to control non--linear interaction terms 
which pair $v$ and $w$ together. 
For example, integrals of the form $\int w^{2}\bar{v}$ and $\int v^{2}\bar{w}$ which appear naturally due to the nonlinearity $G(w,v)$. 
\end{remark}

\begin{remark}
The unconditional uniqueness of NLS \eqref{maineq} with initial data in $H^{s}(\R)$ for $s\geq\frac16$ was first proved by Kato in \cite{TK}. 
\end{remark}
For the proof of Theorem \ref{th1} we will need to localise our functions on the Fourier side and this is achieved through the box operators that are defined as follows: Let $Q_{0}=[-\frac12, \frac12)$ and its translations $Q_{k}=Q_{0}+k$ for all $k\in\mathbb Z$. Consider a partition of unity $\{\sigma_{k}=\sigma_{0}(\cdot-k)\}_{k\in\mathbb Z}\subset C^{\infty}(\mathbb R)$ satisfying 

\begin{itemize}
\item
$ \exists c > 0: \,
\forall \eta \in Q_{0}: \,
|\sigma_{0}(\eta)| \geq c$,
\item
$
\mbox{supp}(\sigma_{0}) \subseteq \{\xi\in\R:|\xi|<1\}$.
\end{itemize}
Note that this implies $1=\sigma_0(0)=\sigma_k(k)$ for all $k\in\Z$. 
Given a partition of unity as above, we define the isometric decomposition operators (box operators)
\begin{equation}
\label{iso}
\Box_{k} \coloneqq \mathcal F^{(-1)} \sigma_{k} \mathcal F, \quad
\left(\forall k \in \Z \right).
\end{equation}
It is not difficult to see that for $1\leq p_{1}\leq p_{2}\leq\infty$ the following holds
\begin{equation}
\label{Bern1}
\|\Box_{k}f\|_{p_{2}}\lesssim\|\Box_{k}f\|_{p_{1}},
\end{equation}
where the implicit constant is independent of $k$ and the function $f$. Having the box operators we may define the modulation spaces $M^{s}_{p,q}(\R), s\in\R, 1\leq p,q\leq\infty$ as 
\begin{equation}
\label{def}
M^{s}_{p,q}(\R)\coloneqq\{f\in S'(\R)\ /\ \|f\|_{M^{s}_{p,q}}\coloneqq\Big(\sum_{k\in\mathbb Z}\langle k\rangle^{sq}\|\Box_{k}f\|_{p}^{q}\Big)^{\frac1{q}}<\infty\},
\end{equation}
with the usual interpretation when the index $q$ is equal to infinity. It can be proved that different choices of the function $\sigma_{0}$ lead to equivalent norms in $M^{s}_{p,q}(\mathbb R)$. 
When $s=0$ we denote the space $M^{0}_{p,q}(\mathbb R)$ by $M_{p,q}(\mathbb R)$. 
In the special case where $p=q=2$ we have $M_{2,2}^{s}(\R)=H^{s}(\R)$.  The usual Sobolev spaces as in \eqref{spacon}. Modulation spaces were introduced by Feichtinger in \cite{FEI}. 
In \cite{CHKP} and \cite{NP} the NLS \eqref{maineq} was studied with initial data $u_{0}\in M^{s}_{p,q}(\R)$ and under the restrictions $s\in[0,\infty), q\in[1,2]$ and $p\in[2,\frac{10q'}{q'+6}),$ existence of weak solutions in the extended sense was proved. 
Moreover, under the extra assumption that $M^{s}_{p,q}(\R)\hookrightarrow L^{3}(\R)$ unconditional well--posedness of the Cauchy problem was shown to be true. 
Unfortunately, the space $M_{\infty,2}(\R)$ is not included in the previously mentioned family of modulation spaces. 
Nevertheless, we are able to obtain an existence result (and uniqueness of solutions under some extra assumptions) for initial data $u_{0}$ in its subspace $H^{s}(\R)+H^{s}(\T)\subset M_{\infty,2}(\R)$ for $s\geq0$.

\subsection{Preliminaries} 
The following lemma will be needed in the proof of Theorem \ref{th1}. It is a straightforward consequence of Young's inequality.

\begin{lemma}
\label{Bern}
Let $1\leq p\leq \infty$ and $\sigma \in C^{\infty}_{c}(\R)$. Then the multiplier operator $T_\sigma: S'(\R) \to S'(\R)$ defined by
\begin{equation*}
(T_\sigma f) = \mathcal F^{-1}(\sigma \cdot \hat{f}), \quad
\forall f \in S'(\R)
\end{equation*}
is bounded on $L^p(\R)$ and
\begin{equation*}
\|T_{\sigma}\|_{L^p(\R)\to L^p(\R)} \lesssim\|\check{\sigma}\|_{L^{1}(\R)}.
\end{equation*}
\end{lemma}
We also need for $S(t)=e^{it\Delta}$, the Schr\"odinger semigroup,  the `conservation of mass'
\begin{equation}
\label{Sch}
\|S(t)f\|_{L^{2}(\R)}=\|f\|_{L^{2}(\R)}.
\end{equation}
Lastly, let us recall the following number theoretic fact (see \cite{HW}, Theorem $315$) which is going to be used throughout the proof of Theorem \ref{th1}: Given an integer $m$, let $d(m)$ denote the number of divisors of $m$. Then 
\begin{equation}
\label{num}
d(m)\lesssim e^{c\frac{\log m}{\log\log m}}=o(m^{\epsilon})
\end{equation}
for all $\epsilon>0$. 

The paper is organised as follows: In Section \ref{smoothsection} we consider initial data $u_{0}=v_{0}+w_{0}$ with $v_{0}, w_{0}$ sufficiently smooth and we show that NLS \eqref{maineq} is locally wellposed. In Section \ref{firststeps} we describe the first steps of the differentiation by parts technique and in Section \ref{treenotation} we define the trees which allow us to continue with the infinite iteration procedure. Finally, in Section \ref{weakextended} we show that the solution $u$ described in Theorem \ref{th1} exists through a smooth approximation procedure and in Section \ref{uniquenessofsol} we prove Theorem \ref{finth2}. 

\end{section}

\begin{section}{smooth initial data}
\label{smoothsection}
Let us assume that the initial data is smooth, that is, $u_{0}=v_{0}+w_{0}$ where $v_{0}\in H^{s_{1}}(\R), w_{0}\in H^{s_{2}}(\T)$ for sufficiently large $s_{1},s_{2}\in\R$. We choose $s_{1}>1, s_{2}=s_{1}+1$. 
Then the spaces $H^{s_{1}}(\R)$ and $H^{s_{2}}(\T)$ are Banach algebras and an easy Banach contraction argument for the operator
\begin{equation}
\label{Banachcontr0}
Tw=e^{it\partial_{x}^{2}}w_{0}\pm\int_{0}^{t}e^{i(t-\tau)\partial_{x}^{2}}|w|^{2}w\ d\tau
\end{equation}
shows that the NLS \eqref{perpe} is locally wellposed in $X_{2}\coloneqq C([0,T],H^{s_{2}}(\T))$ for some $T=T(\|w_{0}\|_{H^{s_{2}}})$. Let $w$ be that solution of NLS \eqref{perpe} in the ball $\{w\in X_{2}:\|w\|_{X_{2}}\leq2\|w_{0}\|_{H^{s_{2}}}\}$ and consider the operator
\begin{equation}
\label{Banachcontr}
Tv=e^{it\partial_{x}^{2}}v_{0}\pm\int_{0}^{t}e^{i(t-\tau)\partial_{x}^{2}}G(w,v)\ d\tau.
\end{equation}
Our goal is to show that $T$ is a contraction in a suitable ball in $X_{1}\coloneqq C([0,T],H^{s_{1}}(\R))$. 

Before we prove this, let us estimate the norm of $\|wv\|_{H^{s_{1}}(\R)}$ for $w\in H^{s_{2}}(\T)$ and $v\in H^{s_{1}}(\R)$. First we need to calculate $\mathcal F(wv)(\xi)$ which equals 
$$\hat{w}\ast\hat{v}(\xi)=\Big(\sum_{n\in\Z}w_{n}\delta_{n}\Big)\ast\hat{v}(\xi)=\sum_{n\in\Z}w_{n}\hat{v}(\xi-n),$$
where we used that for a $1$--periodic function $w$ its Fourier transform is given by $\widehat{w}=\sum_{n\in\Z} w_n \delta_n$, where $\delta_n$ is Dirac delta centered at $n$. Thus,
$$|\mathcal F(wv)(\xi)|^{2}=\sum_{n,m\in\Z}w_{n}\bar{w}_{m}\hat{v}(\xi-n)\overline{\hat{v}(\xi-m)}$$
and, therefore,
\begin{equation*}
	\begin{split}
		\|wv\|_{H^{s_{1}}}^{2}
			&= \int_{\R}(1+|\xi|^{2})^{s_{1}}\sum_{n,m\in\Z}w_{n}\bar{w}_{m}\hat{v}(\xi-n)\overline{\hat{v}(\xi-m)}d\xi \\
			&=\Big|\int_{\R}(1+|\xi|^{2})^{s_{1}}\sum_{n,m\in\Z}w_{n}\bar{w}_{m}\hat{v}(\xi-n)\overline{\hat{v}(\xi-m)}d\xi\Big| \\
			&\leq \sum_{n,m\in\Z}|w_{n}||\bar{w}_{m}|\int_{\R}(1+|\xi|^{2})^{s_{1}}|\hat{v}(\xi-n)||\overline{\hat{v}(\xi-m)}|d\xi. 
	\end{split}
\end{equation*}

For the integral we apply H\"older's inequality
\begin{multline*}
\int_{\R}(1+|\xi|^{2})^{s_{1}}|  \hat{v}(\xi-n)||\overline{\hat{v}(\xi-m)}|d\xi \\
	 \leq\Big(\int_{\R}(1+|\xi|^{2})^{s_{1}}|\hat{v}(\xi-n)|^{2}d\xi\Big)^{\frac12}\Big(\int_{\R}(1+|\xi|^{2})^{s_{1}}|\hat{v}(\xi-m)|^{2}d\xi\Big)^{\frac12},
	\end{multline*}
and this can be estimated from above by the product
$$\lesssim_{s_{1}}(1+|n|^{s_{1}})(1+|m|^{s_{1}})\|v\|_{H^{s_{1}}}^{2},$$
which implies 
$$\|wv\|_{H^{s_{1}}}^{2}\lesssim_{s_{1}}\|v\|_{H^{s_{1}}}^{2}\Big(\sum_{n\in\Z}(1+|n|)^{s_{1}}|w_{n}|\Big)^{2}.$$
Since $s_1>1$, the last sum is again easily estimated using H\"older's inequality as follows
\begin{multline*}
	\Big(\sum_{n\in\Z} (1+|n|)^{s_{1}}  |w_{n}|\Big)^{2}
	 =\Big(\sum_{n\in\Z}\frac{(1+|n|)^{s_{1}+1}}{(1+|n|)}|w_{n}|\Big)^{2}\\
	\leq\Big(\sum_{n\in\Z}(1+|n|)^{2s_{1}+2}|w_{n}|^{2}\Big)\Big(\sum_{n\in\Z}\frac1{(1+|n|)^{2}}\Big)
		\lesssim\|w\|_{H^{s_{1}+1}}^{2}.
\end{multline*}
Thus
\begin{equation}
\label{obvv}
\|wv\|_{H^{s_{1}}(\R)}\lesssim_{s_{1}}\|w\|_{H^{s_{1}+1}(\T)}\|v\|_{H^{s_{1}}(\R)}.
\end{equation}
From \eqref{obvv} and \eqref{nonli} we also obtain  
\begin{equation}
\label{contraction1}
\|Tv\|_{X_{1}}\lesssim\|v_{0}\|_{H^{s_{1}}(\R)}+T\|G(w,v)\|_{X_{1}}\lesssim\|v_{0}\|_{X_{1}}+T(\|v\|_{X_{1}}+\|w\|_{X_{2}})^{3},
\end{equation}
which implies 
\begin{equation}
\label{contraction2}
\|Tv\|_{X_{1}}\lesssim\|v_{0}\|_{H^{s_{1}}(\R)}+T(\|v\|_{X_{1}}+2\|w_{0}\|_{H^{s_{2}}(\T)})^{3}.
\end{equation}
If we assume $v\in B\coloneqq\{v\in X_{1}:\|v\|_{X_{1}}\leq2\|v_{0}\|_{H^{s_{1}}}\coloneqq R\}$, then $T$ maps $B$ into itself for sufficiently small $T=T(\|v_{0}\|_{H^{s_{1}}(\R)}, \|w_{0}\|_{H^{s_{2}}(\T)})$. Indeed, for $T>0$ such that $16T(\|v_{0}\|_{H^{s_{1}}}+\|w_{0}\|_{H^{s_{2}}})^{3}\leq R$, 
we see from \eqref{contraction2} that $Tv\in B$. 

Also, for $v_{1},v_{2}\in B$, it is easy to see 
\begin{equation}
\label{contraction3}
Tv_{1}-Tv_{2}=\pm\int_{0}^{t}\left(G(w,v_{1})-G(w,v_{2})\right)\, d\tau
\end{equation}
where the difference inside the integral equals 
$$|v_{1}|^{2}v_{1}-|v_{2}|^{2}v_{2}+v_{1}^{2}w-v_{2}^{2}w+w^{2}\bar{v}_{1}-w^{2}\bar{v}_{2}+2w|v_{1}|^{2}-2w|v_{2}|^{2}+2v_{1}|w|^{2}-2v_{2}|w|^{2}=$$
$$(v_{1}-v_{2})(|v_{1}|^{2}+\bar{v}_{1}v_{2})+(\bar{v}_{1}-\bar{v}_{2})v_{2}^{2}+w(v_{1}+v_{2})(v_{1}-v_{2})+w^{2}(\bar{v}_{1}-\bar{v}_{2})+$$
$$2w(v_{1}(\bar{v}_{1}-\bar{v}_{2})+\bar{v}_{2}(v_{1}-v_{2}))+2|w|^{2}(v_{1}-v_{2}).$$
Thus,
\begin{equation}
\label{contraction4}
\|Tv_{1}-Tv_{2}\|_{X_{1}}\leq T(\|v_{1}\|_{X_{1}}+\|v_{2}\|_{X_{1}}+\|w\|_{X_{2}})^{2}\|v_{1}-v_{2}\|_{X_{1}},
\end{equation}
which implies, for sufficiently small $T=T(\|v_{0}\|_{H^{s_{1}}}, \|w_{0}\|_{H^{s_{2}}})>0$, that the operator $T:B\to B$ is a contraction. Therefore, we have proved
\begin{lemma}
\label{smoothsol}
Let $s>1$ and $u_{0}=v_{0}+w_{0}\in H^{s}(\R)+H^{s+1}(\T)$. Then NLS \eqref{maineq} is locally wellposed with a solution $u=v+w\in C([0,T],H^{s}(\R))+C([0,T],H^{s+1}(\T))$ where $w$ solves \eqref{perpe} in the sense that satisfies \eqref{Banachcontr0} and $v$ solves \eqref{nonperpe} in the sense that it satisfies \eqref{Banachcontr} for a sufficiently small $T=T(\|v_{0}\|_{H^{s}}, \|w_{0}\|_{H^{s+1}})>0$.
\end{lemma}

\end{section}

\begin{section}{first steps of the iteration process}
\label{firststeps}
From here on, we consider only the case $s_{1}=s_{2}=0$ in Theorem \ref{th1} since for the other cases similar considerations apply. See Remark \ref{reme} at the end of the Section \ref{treenotation} for a more detailed argument. We also assume in the following calculations that the functions $v$ and $w$ are sufficiently smooth. 

Let us define the function $\Phi:\R^{4}\to\R, \Phi(\xi,\xi_{1},\xi_{2},\xi_{3})\coloneqq\xi^{2}-\xi_{1}^{2}+\xi_{2}^{2}-\xi_{3}^{2}$ 
and observe that, 
under the hypothesis $\xi=\xi_{1}-\xi_{2}+\xi_{3}$, it factorizes 
into $\Phi(\xi,\xi_{1},\xi_{2},\xi_{3})=2(\xi-\xi_{1})(\xi-\xi_{3})$. 

By making the change of variables $w\mapsto e^{-it\partial_{x}^{2}}w$, 
we can rewrite the periodic NLS \eqref{perpe} in terms of its Fourier coefficients as 
\begin{align}
\partial_{t}w_{n}
	& =\sum_{n=n_{1}-n_{2}+n_{3}}e^{-2i(n-n_{1})(n-n_{3})t}w_{n_{1}}\bar{w}_{n_{2}}w_{n_{3}}-|w_{n}|^{2}w_{n}+2\Big(\int_{\T}|w|^{2}dx\Big)w_{n} \label{nperpe} \\
	& =\mathscr N_{1}^{t}(w)(n)-\mathscr R_{1}^{t}(w)(n)+\mathscr R_{2}^{t}(w)(n) \nonumber. 
\end{align}

In a similar fashion, we would like to  rewrite the modified NLS 
\eqref{nonperpe}, which contains both periodic and non--periodic 
functions. For this we again make the change of 
variables   $v\mapsto e^{-it\partial_{x}^{2}}v$ and introduce, with the help of the isometric decomposition operators,
$v_{n}\coloneqq\Box_{n}v$ for $n\in\Z$. Note that its Fourier transform, 
$\hat{v}_n$, is a function supported within the interval $(n-1,n+1)$, so, in general, products of the form $\hat{v}_n\hat{v}_m$ can be non-zero only if $|n-m|\le 1$, that is, only neighbouring 
$\hat{v}_n$ can overlap.  
Thus it is convenient to define  
\begin{equation}
\label{approxim}
n\approx m\  \mbox{iff}\ n=m\ \mbox{or}\ n=m+1\ \mbox{or}\ n=m-1 
\end{equation}
 for $n,m\in\Z$. 
Recall that for a $1$--periodic function $w$ its Fourier transform is given by $\widehat{w}=\sum_{n\in\Z} w_n \delta_n$, where $\delta_n$ is Dirac delta centered at $n$. Thus $\Box_n w(x)= w_n e^{inx}$, since the partition of unity we use in the definition of $\Box_n$ 
obeys $1=\sigma_n(n)$. With this we may rewrite the modified NLS 
 \eqref{nonperpe} on the Fourier side, up to constants, as 
\begin{equation}
  \begin{split}\label{nnonperpe}
	\partial_{t}\hat{v}_{n}
	& =E^{1,t}_{I,n}(v_{n_{1}},\bar{v}_{n_{2}},v_{n_{3}})+E^{1,t}_{II,n}(w_{n_{1}},\bar{w}_{n_{2}},v_{n_{3}})+E^{1,t}_{III,n}(w_{n_{1}},\bar{v}_{n_{2}},w_{n_{3}}) \\
	&\phantom{=}\,  +E^{1,t}_{IV,n}(v_{n_{1}},\bar{v}_{n_{2}},w_{n_{3}})+E^{1,t}_{V.n}(v_{n_{1}},\bar{w}_{n_{2}},v_{n_{3}})\, .
  \end{split}
\end{equation} 
where we also introduced  
\begin{equation}
\label{q1}
E^{1,t}_{I,n}(v_{n_{1}},\bar{v}_{n_{2}},v_{n_{3}})(\xi)\coloneqq\!\!\sum_{n\approx n_{1}-n_{2}+n_{3}}\!\!\overbrace{\sigma_{n}(\xi)\iint_{\R^2}e^{-2i(\xi-\xi_{1})(\xi-\xi_{3})t}\hat{v}_{n_{1}}(\xi_{1})\hat{\bar{v}}_{n_{2}}(\xi-\xi_{1}-\xi_{3})\hat{v}_{n_{3}}(\xi_{3})d\xi_{1} d\xi_{3}}^{\eqqcolon\mathcal F(Q^{1,t}_{I,n}(v_{n_{1}},\bar{v}_{n_{2}},v_{n_{3}}))}
\end{equation}
\begin{equation}
\label{q2}
E^{1,t}_{II,n}(w_{n_{1}},\bar{w}_{n_{2}},v_{n_{3}})(\xi)\coloneqq\!\!\sum_{n\approx n_{1}-n_{2}+n_{3}}\!\!\overbrace{\sigma_{n}(\xi)e^{-2i(\xi-n_{1})(n_{1}-n_{2})t}w_{n_{1}}\bar{w}_{n_{2}}\hat{v}_{n_{3}}(\xi-n_{1}+n_{2})}^{\eqqcolon \mathcal F(Q^{1,t}_{II,n}(w_{n_{1}},\bar{w}_{n_{2}},v_{n_{3}}))}
\end{equation}
\begin{equation}
\label{q3}
E^{1,t}_{III,n}(w_{n_{1}},\bar{v}_{n_{2}},w_{n_{3}})(\xi)\coloneqq\!\!\!\sum_{n\approx n_{1}-n_{2}+n_{3}}\!\!\overbrace{\sigma_{n}(\xi)e^{-2i(\xi-n_{1})(\xi-n_{3})t}w_{n_{1}}\hat{\bar{v}}_{n_{2}}(\xi-n_{1}-n_{3})w_{n_{3}}}^{\eqqcolon \mathcal F(Q^{1,t}_{III,n}(w_{n_{1}},\bar{v}_{n_{2}},w_{n_{3}}))}
\end{equation}
\begin{equation}
\label{q4}
E^{1,t}_{IV,n}(v_{n_{1}},\bar{v}_{n_{2}},w_{n_{3}})(\xi)\coloneqq\!\!\sum_{n\approx n_{1}-n_{2}+n_{3}}\!\!\overbrace{\sigma_{n}(\xi)w_{n_{3}}\int_{\R}e^{-2i(\xi-n_{3})(\xi-\xi_{1})t}\hat{v}_{n_{1}}(\xi_{1})\hat{\bar{v}}_{n_{2}}(\xi-n_{3}-\xi_{1})d\xi_{1}}^{\eqqcolon \mathcal F(Q^{1,t}_{IV,n}(v_{n_{1}},\bar{v}_{n_{2}},w_{n_{3}}))}
\end{equation}
\begin{equation}
\label{q5}
E^{1,t}_{V.n}(v_{n_{1}},\bar{w}_{n_{2}},v_{n_{3}})(\xi)\coloneqq\!\!\sum_{n\approx n_{1}-n_{2}+n_{3}}\!\!\overbrace{\sigma_{n}(\xi)\bar{w}_{n_{2}}\int_{\R}e^{-2i(\xi-\xi_{1})(\xi_{1}-n_{2})t}\hat{v}_{n_{1}}(\xi_{1})\hat{v}_{n_{3}}(\xi+n_{2}-\xi_{1})d\xi_{1}}^{\eqqcolon \mathcal F(Q^{1,t}_{V,n}(v_{n_{1}},\bar{w}_{n_{2}},v_{n_{3}}))}.
\end{equation}
\begin{remark}
A short note on our notation is necessary here: The expression  $E^{1,t}_{I,n}(v_{n_{1}},\bar{v}_{n_{2}},v_{n_{3}})$ above depends not only on the single $v_{n_{1}}$, $\bar{v}_{n_{2}}$, or $v_{n_{3}}$, but on the sequences $(v_{n_{1}})_{n_1\in\Z}$, $(\bar{v}_{n_{2}})_{n_2\in\Z}$, and  $(v_{n_{3}})_{n_3\in\Z}$. So one should instead write  $E^{1,t}_{I,n}((v_{n_{1}})_{n_1\in\Z},(\bar{v}_{n_{2}})_{n_2\in\Z},(v_{n_{3}})_{n_3\in\Z})$, or simply, 
 $E^{1,t}_{I,n}(v,\bar{v},v)$. 
 However, when we construct a tree--type expansion later, it will be very important to know in which order the $v_n$ and $w_m$ appear in 
 considerably more involved expressions. Thus it will be convenient to  write $E^{1,t}_{I,n}(v,\bar{v},v)$ as $E^{1,t}_{I,n}(v_{n_{1}},\bar{v}_{n_{2}},v_{n_{3}})$, keeping in mind, that one sums over $n_1$, $n_2$, and $n_3$.  The same applies to the other terms on the right--hand side of  equation \eqref{nnonperpe}.	
\end{remark}

\begin{remark}
\label{oll}
The operator $Q^{1,t}_{I,n}$ in the definition of $E^{1,t}_{I,n}$ in equation \eqref{q1} is the same as the operator $Q^{1,t}_{n}$ 
studied in \cite{CHKP} and \cite{NP}. Here let us notice that if 
we choose functions such that 
$\hat{v}_{n_{1}}=w_{n_{1}}\delta_{n_{1}}$ and $\hat{v}_{n_{2}}=w_{n_{2}}\delta_{n_{2}}$ then we obtain the relation 
$Q^{1,t}_{I,n}(v_{n_{1}},\bar{v}_{n_{2}},v_{n_{3}})=Q^{1,t}_{II,n}(w_{n_{1}},\bar{w}_{n_{2}},v_{n_{3}})$. 
Similar relations hold between $Q^{1,t}_{I,n}$ and the remaining operators $Q^{1,t}_{III,n},Q^{1,t}_{IV,n}$ and $Q^{1,t}_{V,n}$.
\end{remark}
We split the sums in \eqref{q1}, \eqref{q2}, \eqref{q3}, \eqref{q4} and \eqref{q5} into 
$$\sum_{\substack{n_{1}\approx n\\ or\\ n_{3}\approx n}}\ldots+\sum_{n_{1}, n_{3}\not\approx n}\ldots $$
and define the resonant operators
\begin{equation}
\label{main9}
R^{t}_{2}(v)(n)\coloneqq\Big(\sum_{n_{1}\approx n}+\sum_{n_{3}\approx n}\Big)\Big(Q^{1,t}_{I,n}+Q^{1,t}_{II,n}+Q^{1,t}_{III,n}+Q^{1,t}_{IV,n}+Q^{1,t}_{V,n}\Big)
\end{equation}
$$R^{t}_{1}(v)(n)\coloneqq\sum_{\substack{n_{1}\approx n\\ and\\ n_{3}\approx n}}\Big(Q^{1,t}_{I,n}+Q^{1,t}_{II,n}+Q^{1,t}_{III,n}+Q^{1,t}_{IV,n}+Q^{1,t}_{V,n}\Big) $$
and the non-resonant operator 
\begin{equation}
\label{main10}
N_{1}^{t}(v)(n)\coloneqq\sum_{n_{1}, n_{3}\not\approx n}\Big(Q^{1,t}_{I,n}+Q^{1,t}_{II,n}+Q^{1,t}_{III,n}+Q^{1,t}_{IV,n}+Q^{1,t}_{V,n}\Big).
\end{equation}
With this notation, equation \eqref{nnonperpe} can be written in the form
\begin{equation}
\label{mainmain}
\partial_{t}v_{n}=R^{t}_{2}(v)(n)-R^{t}_{1}(v)(n)+N_{1}^{t}(v)(n),
\end{equation}
keeping in mind that the operators appearing in the RHS above depend also on the periodic function $w$, which we suppress in our notation, for simplicity. For the resonant part we have the estimate

\begin{lemma}
\label{lem}
For $j=1,2$ 
$$\|R^{t}_{j}(v)\|_{l^{2}(\Z)L^{2}(\R)}\lesssim\|v\|^{3}_{L^{2}(\R)}+\|w\|_{L^{2}(\T)}\|v\|^{2}_{L^{2}(\R)}+\|w\|^{2}_{L^{2}(\T)}\|v\|_{L^{2}(\R)}$$
and
$$\|R^{t}_{j}(v)-R^{t}_{j}(u)\|_{l^{2}(\Z)L^{2}(\R)}\lesssim$$
$$(\|v\|^{2}_{L^{2}(\R)}+\|u\|^{2}_{L^{2}(\R)}+\|w\|_{L^{2}(\T)}(\|v\|_{L^{2}(\R)}+\|u\|_{L^{2}(\R)})+\|w\|^{2}_{L^{2}(\T)})\|v-u\|_{L^{2}(\R)}.$$
\end{lemma}
\begin{proof}
Both resonant operators contain a sum that only involves the $v$ function, that is 
$$\Big(\sum_{n_{1}\approx n}+\sum_{n_{3}\approx n}-\sum_{\substack{n_{1}\approx n\\ and\\ n_{3}\approx n}}\Big)Q^{1,t}_{I,n}(v_{n_{1}},\bar{v}_{n_{2}},v_{n_{3}}).$$
As mentioned in Remark \ref{oll} this operator was estimated in \cite{NP}, it gives the upper bound of $\|v\|^{3}_{2}$ and we refer the interested reader to Lemma $10$ of that paper. 

For the sum that contains $Q^{1,t}_{II,n}$ and $Q^{1,t}_{III,n}$ it suffices to estimate only $Q^{1,t}_{II,n}$., the bound for the sum involving $Q^{1,t}_{III,n}$ is very similar to one for $Q^{1,t}_{II,n}$  Moreover, since, for fixed $n\in\Z$, the sum 
$$\sum_{\substack{n_{1}\approx n\\ and\\ n_{3}\approx n}}Q^{1,t}_{II,n}(w_{n_{1}},\bar{w}_{n_{2}},v_{n_{3}})$$
is only over the neighbours of $n$, we only look at the part where $n_{1}=n$ and $n_{3}=n$, the other summands are bounded in the same way. Then we have the estimate
$$\Big\|\sigma_{n}(\xi)w_{n}\bar{w}_{n}\hat{v}_{n}(\xi)\Big\|_{l^{2}(\Z)L_\xi^{2}(\R)}\lesssim\Big(\sum_{n\in\Z}|w_{n}|^{4}\|v_{n}\|_{2}^{2}\Big)^{\frac12}\lesssim \|w_{n}\|^{2}_{l^{\infty}(\Z)}\|v\|_{2}\leq\|w\|_{L^{2}(\T)}^{2}\|v\|_{L^{2}(\R)}, $$
by the embedding $l^{2}(\Z)\hookrightarrow l^{\infty}(\Z)$. To continue it suffices to look at the sum 
$$\sum_{n_{1}\approx n}Q^{1,t}_{II,n}(w_{n_{1}},\bar{w}_{n_{2}},v_{n_{3}}).$$
Again, since it consists of finitely many summands, depending on whether $n_{1}=n-1$ or $n_{1}=n$ or $n_{1}=n+1,$ it is enough to estimate the part where $n_{1}=n$. In this case, we have
$$\Big\|\sigma_{n}(\xi)w_{n}\sum_{n_{2}\in\Z}e^{-2it(\xi-n)(n-n_{2})}\bar{w}_{n_{2}}\hat{v}_{n_{2}}(\xi-n+n_{2})\Big\|_{L^{2}(\R)}\lesssim|w_{n}|\sum_{n_{2}\in\Z}|w_{n_{2}}|\|v_{n_{2}}\|_{2},$$
so with H\"older's inequality we get the upper bound
$$|w_{n}|\Big(\sum_{n_{2}\in\Z}|w_{n_{2}}|^{2}\Big)^{\frac12}\Big(\sum_{n_{2}\in\Z}\|v_{n_{2}}\|_{2}^{2}\Big)^{\frac12}=|w_{n}|\|w\|_{L^{2}(\T)}\|v\|_{L^{2}(\R)}.$$
Taking the $l^{2}(\Z)$ norm we obtain 
$$\|w\|_{L^{2}(\T)}^{2}\|v\|_{L^{2}(\R)}.$$

For the sum that contains $Q^{1,t}_{IV,n}$ and $Q^{1,t}_{V.n}$ it suffices to estimate only $Q^{1,t}_{V,n}$. As before, from the sum 
$$\sum_{\substack{n_{1}\approx n\\ and\\ n_{3}\approx n}}Q^{1,t}_{V,n}(v_{n_{1}},\bar{w}_{n_{2}},v_{n_{3}})$$
we may look only at the part where $n_{1}=n$ and $n_{3}=n$. Thus, we have
$$\Big\|\sigma_{n}(\xi)\bar{w}_{n}\int_{\R}e^{-2it(\xi-\xi_{1})(\xi_{1}-n)}\hat{v}_{n}(\xi_{1})\hat{v}_{n}(\xi-\xi_{1}+n)d\xi_{1}\Big\|_{L^{2}(\R)}$$
which, by setting $\hat{V}_{n}=e^{it\xi^{2}}\hat{v}_{n}$ and using \eqref{Sch}, we may rewrite as 
$$\Big\|\sigma_{n}(\xi)\bar{w}_{n}e^{-itn^{2}}\int_{\R}\hat{V}_{n}(\xi_{1})\hat{V}_{n}(\xi-\xi_{1}+n)d\xi_{1}\Big\|_{L^{2}(\R)}\lesssim|w_{n}|\Big\|\hat{V}_{n}\ast\hat{V}_{n}(\cdot+n)\Big\|_{L^{2}(\R)}.$$
The last expression equals
$$|w_{n}|\Big\|V_{n}e^{in(\cdot)}V_{n}\Big\|_{L^{2}(\R)}=|w_{n}|\|V_{n}\|^{2}_{L^{4}(\R)}\lesssim|w_{n}|\|V_{n}\|_{L^{2}(\R)}^{2},$$
where we used \eqref{Bern1} and, since $\|V_{n}\|_{2}=\|v_{n}\|_{2}$ (by \eqref{Sch}), we can take the $l^{2}(\Z)$ norm in $n$ and obtain the upper bound
$$\Big(\sum_{n}|w_{n}|^{2}\|v_{n}\|_{2}^{4}\Big)^{\frac12}\leq\|w\|_{L^{2}(\T)}\|\{\|v_{n}\|_{2}\}_{n\in\Z}\|_{l^{\infty}(\Z)}^{2}\leq\|w\|_{L^{2}(\T)}\|v\|_{L^{2}(\R)}^{2},$$
by the embedding $l^{2}(\Z)\hookrightarrow l^{\infty}(\Z)$. Finally, we look at the sum 
$$\sum_{n_{1}\approx n}Q^{1,t}_{V,n}(v_{n_{1}},\bar{w}_{n_{2}},v_{n_{3}}).$$
As before, it suffices to look at the term where $n_{1}=n$. In this case we have
$$\Big\|\sigma_{n}(\xi)\sum_{n_{2}\in\Z}\bar{w}_{n_{2}}\int_{\R}e^{-2it(\xi-\xi_{1})(\xi_{1}-n_{2})}\hat{v}_{n_{1}}(\xi_{1})\hat{v}_{n_{2}}(\xi-\xi_{1}+n_{2})d\xi_{1}\Big\|_{L^{2}(\R)},
$$
and setting again $\hat{V}_{n}=e^{it\xi^{2}}\hat{v}_{n}$, we arrive at the upper bound
$$\sum_{n_{2}\in\Z}|w_{n_{2}}|\|\hat{V}_{n}\ast\hat{V}_{n_{2}}(\cdot+n_{2})\|_{2}=\sum_{n_{2}\in\Z}|w_{n_{2}}|\|V_{n}e^{in_{2}(\cdot)}V_{n_{2}}\|_{2}=\sum_{n_{2}\in\Z}|w_{n_{2}}|\|V_{n}V_{n_{2}}\|_{2}.$$
Applying H\"older's inequality, \eqref{Bern1} and \eqref{Sch} we continue the estimate as follows
$$\sum_{n_{2}\in\Z}|w_{n_{2}}|\|V_{n}\|_{4}\|V_{n_{2}}\|_{4}\lesssim\|V_{n}\|_{2}\sum_{n_{2}\in\Z}|w_{n_{2}}|\|V_{n_{2}}\|_{2}=\|v_{n}\|_{2}\sum_{n_{2}\in\Z}|w_{n_{2}}|\|v_{n_{2}}\|_{2}$$
$$\leq\|v_{n}\|_{2}\Big(\sum_{n_{2}\in\Z}|w_{n_{2}}|^{2}\Big)^{\frac12}\Big(\sum_{n_{2}\in\Z}\|v_{n_{2}}|^{2}\Big)^{\frac12}=\|v_{n}\|_{2}\|w\|_{L^{2}(\T)}\|v\|_{L^{2}(\R)}.$$
Taking the $l^{2}(\Z)$ norm in $n$ finishes the proof.
\end{proof}
\begin{remark}
\label{outoftheblue}
In \cite{GKO} it was proved that the resonant part of the periodic solution $w$ satisfies 
$$\|\mathscr R_{j}^{t}(w)\|_{L^{2}(\T)}\lesssim\|w\|^{3}_{L^{2}(\T)}$$
for $j=1,2$. This will be used later in Lemma \ref{finaal} for the estimate of the $N_{r}^{(J)}$ operator. 
\end{remark}
In order to continue the iteration process we define the sets
\begin{equation}
\label{set1}
A_{N}(n)=\{(n_{1},n_{2},n_{3})\in\mathbb Z^3:n_{1}-n_{2}+n_{3}\approx n, n_{1}\not\approx n\not\approx n_{3}, |\Phi(n,n_{1},n_{2},n_{3})|\leq N\}
\end{equation}
and
\begin{equation}
\label{idid}
A_{N}(n)^{c}=\{(n_{1},n_{2},n_{3})\in\mathbb Z^3:n_{1}-n_{2}+n_{3}\approx n, n_{1}\not\approx n\not\approx n_{3}, |\Phi(n,n_{1},n_{2},n_{3})|> N\}.
\end{equation}
The number $N>0$ is considered to be large and will be fixed later 
in the proof. The non-resonant operator $N_{1}^{t}$ we 
split as
\begin{equation}
\label{main13}
N_{1}^{t}(v)(n)=N_{11}^{t}(v)(n)+N_{12}^{t}(v)(n),
\end{equation}
where 
\begin{align*}
N_{11}^{t}(v)(n)
	&=\sum_{A_{N}(n)}\Big(Q^{1,t}_{I,n}(v_{n_{1}},\bar{v}_{n_{2}},v_{n_{3}})+Q^{1,t}_{II,n}(w_{n_{1}},\bar{w}_{n_{2}},v_{n_{3}})+Q^{1,t}_{III,n}(w_{n_{1}},\bar{v}_{n_{2}},w_{n_{3}})\\ 
	&\phantom{=\sum_{A_{N}(n)}\Big( ~}+Q^{1,t}_{IV,n}(v_{n_{1}},\bar{v}_{n_{2}},w_{n_{3}})+Q^{1,t}_{V,n}(v_{n_{1}},\bar{w}_{n_{2}},v_{n_{3}})\Big),
\end{align*}
and the following yields a convenient bound on $N^t_{11}$.

\begin{lemma}
\label{lemle}
$$\|N_{11}^{t}(v)\|_{l^{2}(\Z)L^{2}(\R)}\lesssim N^{\frac12+}(\|v\|^{3}_{L^{2}(\R)}+\|w\|_{L^{2}(\T)}\|v\|^{2}_{L^{2}(\R)}+\|w\|^{2}_{L^{2}(\T)}\|v\|_{L^{2}(\R)})$$
and
$$\|N_{11}^{t}(v)-N_{11}^{t}(u)\|_{l^{2}(\Z)L^{2}(\R)}\lesssim$$ 
$$N^{\frac12+}(\|v\|^{2}_{L^{2}(\R)}+\|u\|^{2}_{L^{2}(\R)}+\|w\|_{L^{2}(\T)}(\|v\|_{L^{2}(\R)}+\|u\|_{L^{2}(\R)})+\|w\|^{2}_{L^{2}(\T)})\|v-u\|_{L^{2}(\R)}.$$
\end{lemma}
\begin{proof}
The part
$$\sum_{A_{N}(n)}Q^{1,t}_{I,n}(v_{n_{1}},\bar{v}_{n_{2}},v_{n_{3}})$$
has been estimated in \cite{NP}, Lemma $11$ giving an upper bound of the form $N^{\frac12+}\|v\|_{2}^{3}$. 
For the sum that contains $Q^{1,t}_{II,n}$ and $Q^{1,t}_{III,n}$ it suffices to estimate only $Q^{1,t}_{II,n}$, the other one being similar. We have
$$\sum_{A_{N}(n)}\Big\|\sigma_{n}(\xi)e^{-2it(\xi-n_{1})(n_{1}-n_{2})}w_{n_{1}}\bar{w}_{n_{2}}\hat{v}_{n_{3}}(\xi-n_{1}+n_{2})\Big\|_{L^{2}(\R)}\lesssim\sum_{A_{N}(n)}|w_{n_{1}}||w_{n_{2}}|\|v_{n_{3}}\|_{2},$$
which by H\"older's inequality implies the estimate
\begin{equation}
\label{ghg}
\Big(\sum_{A_{N}(n)}1\Big)^{\frac12}\Big(\sum_{A_{N}(n)}|w_{n_{1}}|^{2}|w_{n_{2}}|^{2}\|v_{n_{3}}\|_{2}^{2}\Big)^{\frac12}.
\end{equation}
The first factor is estimated by $N^{\frac12+}$ with the use of \eqref{num} and then, by taking the $l^{2}(\Z)$ norm of 
the second sum and applying Young's inequality in $l^{1}(\Z)$, 
we obtain the upper bound 
$$N^{\frac12+}\Big(\sum_{n\in\Z}\sum_{A_{N}(n)}|w_{n_{1}}|^{2}|w_{n_{2}}|^{2}\|v_{n_{3}}\|_{2}^{2}\Big)^{\frac12}\leq N^{\frac12+}\|w\|_{L^{2}(\T)}^{2}\|v\|_{L^{2}(\R)}.$$ 

For the sum that contains $Q^{1,t}_{IV.n}$ and $Q^{1,t}_{V,n}$ it 
again suffices to estimate only $Q^{1,t}_{V,n}$. In this case, 
letting $\hat{V}_{n}=e^{it\xi^{2}}\hat{v}_{n}$, we have 
$$\sum_{A_{N}(n)}\Big\|\sigma_{n}(\xi)\bar{w}_{n_{2}}\int_{\R}e^{-2it(\xi-\xi_{1})(\xi_{1}-n_{2})}\hat{v}_{n_{1}}(\xi_{1})\hat{v}_{n_{3}}(\xi-\xi_{1}+n_{2})d\xi_{1}\Big\|_{L^{2}(\R)}\lesssim$$
$$\sum_{A_{N}(n)}|w_{n_{2}}|\Big\|\hat{V}_{n_{1}}\ast\hat{V}_{n_{3}}(\cdot+n_{2})\Big\|_{L^{2}(\R)}=\sum_{A_{N}(n)}|w_{n_{2}}|\Big\|V_{n_{1}}V_{n_{3}}\Big\|_{L^{2}(\R)}\leq$$
$$\sum_{A_{N}(n)}|w_{n_{2}}|\|V_{n_{1}}\|_{4}\|V_{n_{3}}\|_{4}\lesssim\sum_{A_{N}(n)}|w_{n_{2}}|\|V_{n_{1}}\|_{2}\|V_{n_{3}}\|_{2}=\sum_{A_{N}(n)}|w_{n_{2}}|\|v_{n_{1}}\|_{2}\|v_{n_{3}}\|_{2},$$
where we used \eqref{Bern1} and \eqref{Sch}. Then the estimate continues as in \eqref{ghg} giving the upper bound 
$N^{\frac12+}\|w\|_{2}\|v\|_{2}^{2}$. 
\end{proof}
For the $N_{12}^{t}$ operator we only look at frequencies where $|\Phi(n,n_{1},n_{2},n_{3})|>N$, 
which means that we can apply the differentiation by parts techniques, in order to take advantage of possible cancellations, due to the 
 fact that the exponential terms contain the phase factor $\Phi(n,n_{1},n_{2},n_{3})$, having a large magnitude. 
 By doing this separately to the $Q^{1,t}_{I,n}\ldots, Q^{1,t}_{V,n}$ operators we obtain the following expressions 
\begin{equation}
\begin{split}
\label{ttr1}
&\partial_{t}\Big(\overbrace{\sigma_{n}(\xi)\int_{\mathbb R^2}\frac{e^{-2it(\xi-\xi_{1})(\xi-\xi_{3})}}{-2i(\xi-\xi_{1})(\xi-\xi_{3})}\ \hat{v}_{n_{1}}(\xi_{1})\hat{\bar{v}}_{n_{2}}(\xi-\xi_{1}-\xi_{3})\hat{v}_{n_{3}}(\xi_{3})\ d\xi_{1}d\xi_{3}}^{\eqqcolon \mathcal F(\tilde{Q}^{1,t}_{I,n})}\Big) \\
&- \, \overbrace{\sigma_{n}(\xi)\int_{\mathbb R^2}\frac{e^{-2it(\xi-\xi_{1})(\xi-\xi_{3})}}{-2i(\xi-\xi_{1})(\xi-\xi_{3})}\partial_{t}\Big(\hat{v}_{n_{1}}(\xi_{1})\hat{\bar{v}}_{n_{2}}(\xi-\xi_{1}-\xi_{3})\hat{v}_{n_{3}}(\xi_{3})\Big)\ d\xi_{1}d\xi_{3}}^{\eqqcolon \mathcal F(T^{1,t}_{I,n})},
\end{split}
\end{equation}
\begin{equation}
\begin{split}
\label{ttr2}
&\partial_{t}\Big(\overbrace{\sigma_{n}(\xi)\ \frac{e^{-2it(\xi-n_{1})(n_{1}-n_{2})}}{-2i(\xi-n_{1})(n_{1}-n_{2})}\ w_{n_{1}}\bar{w}_{n_{2}}\hat{v}_{n_{3}}(\xi-n_{1}+n_{2})}^{\eqqcolon \mathcal F(\tilde{Q}^{1,t}_{II,n})}\Big) \\
&- \, \overbrace{\sigma_{n}(\xi)\ \frac{e^{-2it(\xi-n_{1})(n_{1}-n_{2})}}{-2i(\xi-n_{1})(n_{1}-n_{2})}\ \partial_{t}\Big(w_{n_{1}}\bar{w}_{n_{2}}\hat{v}_{n_{3}}(\xi-n_{1}+n_{2})\Big)}^{\eqqcolon \mathcal F(T^{1,t}_{II,n})},
\end{split}
\end{equation}
\begin{equation}
\begin{split}
\label{ttr3}
&\partial_{t}\Big(\overbrace{\sigma_{n}(\xi)\ \frac{e^{-2it(\xi-n_{1})(\xi-n_{3})}}{-2i(\xi-n_{1})(\xi-n_{3})}\ w_{n_{1}}\hat{\bar{v}}_{n_{2}}(\xi-n_{1}-n_{3})w_{n_{3}}}^{\eqqcolon \mathcal F(\tilde{Q}^{1,t}_{III,n})}\Big)\\
&- \, \overbrace{\sigma_{n}(\xi)\ \frac{e^{-2it(\xi-n_{1})(\xi-n_{3})}}{-2i(\xi-n_{1})(\xi-n_{3})}\ \partial_{t}\Big(w_{n_{1}}\hat{\bar{v}}_{n_{2}}(\xi-n_{1}-n_{3})w_{n_{3}}\Big)}^{\eqqcolon \mathcal F(T^{1,t}_{III,n})}
\end{split}
\end{equation}

\begin{equation}
\begin{split}
\label{ttr4}
&\partial_{t}\Big(\overbrace{\sigma_{n}(\xi)w_{n_{3}}\ \int_{\R}\frac{e^{-2it(\xi-n_{3})(\xi-\xi_{1})}}{-2i(\xi-n_{3})(\xi-\xi_{1})}\ \hat{v}_{n_{1}}(\xi_{1})\hat{\bar{v}}_{n_{2}}(\xi-\xi_{1}-n_{3})d\xi_{1}}^{\eqqcolon \mathcal F(\tilde{Q}^{1,t}_{IV.n})}\Big) \\
&- \, \overbrace{\sigma_{n}(\xi)\ \int_{\R}\frac{e^{-2it(\xi-n_{3})(\xi-\xi_{1})}}{-2i(\xi-n_{3})(\xi-\xi_{1})}\ \partial_{t}\Big(\hat{v}_{n_{1}}(\xi_{1})\hat{\bar{v}}_{n_{2}}(\xi-\xi_{1}-n_{3})w_{n_{3}}\Big)d\xi_{1}}^{\eqqcolon \mathcal F(T^{1,t}_{IV,n})}
\end{split}
\end{equation}
and
\begin{equation}
\begin{split}
\label{ttr5}
&\partial_{t}\Big(\overbrace{\sigma_{n}(\xi)\bar{w}_{n_{2}}\ \int_{\R}\frac{e^{-2it(\xi-\xi_{1})(\xi_{1}-n_{2})}}{-2i(\xi-\xi_{1})(\xi_{1}-n_{2})}\ \hat{v}_{n_{1}}(\xi_{1})\hat{v}_{n_{3}}(\xi-\xi_{1}+n_{2})d\xi_{1}}^{\eqqcolon \mathcal F(\tilde{Q}^{1,t}_{V,n})}\Big)- \\
&\overbrace{\sigma_{n}(\xi)\ \int_{\R}\frac{e^{-2it(\xi-\xi_{1})(\xi_{1}-n_{2})}}{-2i(\xi-\xi_{1})(\xi_{1}-n_{2})}\ \partial_{t}\Big(\hat{v}_{n_{1}}(\xi_{1})\bar{w}_{n_{2}}\hat{v}_{n_{3}}(\xi-\xi_{1}+n_{2})\Big)d\xi_{1}}^{\mathcal F(T^{1,t}_{V,n})}.
\end{split}
\end{equation}
This allows us to express 
\begin{equation}
\begin{split}
\label{nex}
N_{12}^{t}(v)
&= \sum_{A_{N}(n)}\Big(Q^{1,t}_{I,n}+Q^{1,t}_{II,n}+Q^{1,t}_{III,n}+Q^{1,t}_{IV,n}+Q^{1,t}_{V,n}\Big) \\
&= \partial_{t}\Big(\overbrace{\sum_{A_{N}(n)}\Big(\tilde{Q}^{1,t}_{I,n}+\tilde{Q}^{1,t}_{II,n}+\tilde{Q}^{1,t}_{III,n}+\tilde{Q}^{1,t}_{IV,n}+\tilde{Q}^{1,t}_{V,n}\Big)}^{\eqqcolon N_{21}^{t}(v)}\Big) \\      
&\phantom{+~}+\, \overbrace{\sum_{A_{N}(n)}\Big(T^{1,t}_{I,n}+T^{1,t}_{II,n}+T^{1,t}_{III,n}+T^{1,t}_{IV,n}+T^{1,t}_{V,n}\Big)}^{\eqqcolon N_{22}^{t}(v)}.
\end{split}
\end{equation}
At this point let us also define the operators 
\begin{align}
\label{rr1}
\mathcal F(R^{1,t}_{I,n}(v_{n_{1}},\bar{v}_{n_{2}},v_{n_{3}}))(\xi)
	&= \sigma_{n}(\xi)\int_{\mathbb R^2}\frac{\hat{v}_{n_{1}}(\xi_{1})\hat{\bar{v}}_{n_{2}}(\xi-\xi_{1}-\xi_{3})\hat{v}_{n_{3}}(\xi_{3})}{(\xi-\xi_{1})(\xi-\xi_{3})}\ d\xi_{1}d\xi_{3}\, , 
		\\ 
\label{rr2}
\mathcal F(R^{1,t}_{II,n}(w_{n_{1}},\bar{w}_{n_{2}},v_{n_{3}}))(\xi)
	& =\sigma_{n}(\xi)\ \frac{w_{n_{1}}\bar{w}_{n_{2}}\hat{v}_{n_{3}}(\xi-n_{1}+n_{2})}{(\xi-n_{1})(n_{1}-n_{2})} \, , \\
\label{rr3}
\mathcal F(R^{1,t}_{III,n}(w_{n_{1}},\bar{v}_{n_{2}},w_{n_{3}}))(\xi)
	&= \sigma_{n}(\xi)\ \frac{w_{n_{1}}\hat{\bar{v}}_{n_{2}}(\xi-n_{1}-n_{3})w_{n_{3}}}{(\xi-n_{1})(\xi-n_{3})} \, , \\
\label{rr4}
\mathcal F(R^{1,t}_{IV,n}(v_{n_{1}},\bar{v}_{n_{2}},w_{n_{3}}))(\xi)
	&= \sigma_{n}(\xi)w_{n_{3}}\int_{\R}\frac{\hat{v}_{n_{1}}(\xi_{1})\hat{\bar{v}}_{n_{2}}(\xi-\xi_{1}-n_{3})}{(\xi-n_{3})(\xi-\xi_{1})}\ d\xi_{1}\, ,  \\
\label{rr5}
\mathcal F(R^{1,t}_{V,n}(v_{n_{1}},\bar{w}_{n_{2}},v_{n_{3}}))(\xi)
	&= \sigma_{n}(\xi)\bar{w}_{n_{2}}\int_{\R}\frac{\hat{v}_{n_{1}}(\xi_{1})\hat{v}_{n_{3}}(\xi-\xi_{1}+n_{2})}{(\xi-\xi_{1})(\xi_{1}-n_{2})}\ d\xi_{1} \, ,
\end{align}
and observe that, if we let 
\begin{equation}
\label{namee}
\hat{V}_{n}=e^{it\xi^{2}}\hat{v}_{n},\ W_{n}=e^{itn^{2}}w_{n}\, ,
\end{equation} 
then 
\begin{align}
\label{rela1}
\mathcal F(\tilde{Q}^{1,t}_{I,n}(v_{n_{1}},\bar{v}_{n_{2}},v_{n_{3}}))(\xi)
	&= e^{-it\xi^{2}}\mathcal F(R^{1,t}_{I,n}(V_{n_{1}},\bar{V}_{n_{2}},V_{n_{3}}))(\xi) \, , \\
\label{rela2}
\mathcal F(\tilde{Q}^{1,t}_{II,n}(w_{n_{1}},\bar{w}_{n_{2}},v_{n_{3}}))(\xi)
	&= e^{-it\xi^{2}}\mathcal F(R^{1,t}_{II,n}(W_{n_{1}},\bar{W}_{n_{2}},V_{n_{3}}))(\xi) \, ,\\
\label{rela3}
\mathcal F(\tilde{Q}^{1,t}_{III,n}(w_{n_{1}},\bar{v}_{n_{2}},w_{n_{3}}))(\xi)
	&= e^{-it\xi^{2}}\mathcal F(R^{1,t}_{III,n}(W_{n_{1}},\bar{V}_{n_{2}},W_{n_{3}}))(\xi) \, ,\\
\label{rela4}
\mathcal F(\tilde{Q}^{1,t}_{IV,n}(v_{n_{1}},\bar{v}_{n_{2}},w_{n_{3}}))(\xi)
	&= e^{-it\xi^{2}}\mathcal F(R^{1,t}_{IV,n}(V_{n_{1}},\bar{V}_{n_{2}},W_{n_{3}}))(\xi) \\
\label{rela5}
\mathcal F(\tilde{Q}^{1,t}_{V,n}(v_{n_{1}},\bar{w}_{n_{2}},v_{n_{3}}))(\xi)
	& =e^{-it\xi^{2}}\mathcal F(R^{1,t}_{V,n}(V_{n_{1}},\bar{W}_{n_{2}},V_{n_{3}}))(\xi). 
\end{align}
Also notice that, writing out the Fourier transforms of the functions inside the integral of \eqref{rr1}, it is not difficult to see 
\begin{equation}
\label{main7}
R^{1,t}_{I,n}(v_{n_{1}},\bar{v}_{n_{2}},v_{n_{3}})(x)=\int_{\mathbb R^3}K^{(1)}_{n}(x,x_{1},y,x_{3})v_{n_{1}}(x)\bar{v}_{n_{2}}(y)v_{n_{3}}(x_{3})\ dx_{1}dydx_{3},
\end{equation}
where
$$K^{(1)}_{n}(x,x_{1},y,x_{3})=\int_{\mathbb R^3}e^{i\xi_{1}(x-x_{1})+i\eta(x-y)+i\xi_{3}(x-x_{3})}\ \frac{\sigma_{n}(\xi_{1}+\eta+\xi_{3})}{(\eta+\xi_{1})(\eta+\xi_{3})}\ d\xi_{1}d\eta d\xi_{3}=$$
$$\mathcal F^{-1}\rho^{(1)}_{n}(x-x_{1},x-y,x-x_{3})$$
and 
$$\rho_{n}^{(1)}(\xi_{1},\eta,\xi_{3})=\frac{\sigma_{n}(\xi_{1}+\eta+\xi_{3})}{(\eta+\xi_{1})(\eta+\xi_{3})}.$$

\begin{remark}
\label{asbefore}
The operators $\tilde{Q}^{1,t}_{I,n}$ and $R^{1,t}_{I,n}$ are the same as the operators $\tilde{Q}^{1,t}_{n}$ and $R^{1,t}_{n}$ studied in \cite[Lemma 12]{CHKP} and \cite[Lemma 12]{NP}. Also notice that for $\hat{v}_{n_{2}}=w_{n_{2}}\delta_{n_{2}}$ and $\hat{v}_{n_{1}}=w_{n_{1}}\delta_{n_{1}}$ we have $R^{1,t}_{I,n}(v_{n_{1}},\bar{v}_{n_{2}},v_{n_{3}})=R^{1,t}_{II,n}(w_{n_{1}},\bar{w}_{n_{2}},v_{n_{3}})$. Similar relations hold between $R^{1,t}_{I,n}$ and the remaining operators $R^{1,t}_{III,n}, R^{1,t}_{IV,n}$ and $R^{1,t}_{V,n}$. 
\end{remark}

\begin{lemma} \label{fir}
 For fixed $n, n_1, n_2, n_3$, the multilinear operators defined in \eqref{rr1}--\eqref{rr5} are bounded by 
\begin{align*}
\|R^{1,t}_{I,n}(v_{n_{1}},\bar{v}_{n_{2}},v_{n_{3}})\|_{L^{2}(\R)}&\lesssim  \frac{\|v_{n_{1}}\|_{2}\|v_{n_{2}}\|_{2}\|v_{n_{3}}\|_{2}}{|n-n_{1}||n-n_{3}|},\\
\|R^{1,t}_{II,n}(w_{n_{1}},\bar{w}_{n_{2}},v_{n_{3}})\|_{L^{2}(\R)}&\lesssim  \frac{|w_{n_{1}}||w_{n_{2}}|\|v_{n_{3}}\|_{2}}{|n-n_{1}||n-n_{3}|},\\
\|R^{1,t}_{III,n}(w_{n_{1}},\bar{v}_{n_{2}},w_{n_{3}})\|_{L^{2}(\R)}&\lesssim  \frac{|w_{n_{1}}|\|v_{n_{2}}\|_{2}|w_{n_{3}}|}{|n-n_{1}||n-n_{3}|},\\
\|R^{1,t}_{IV,n}(v_{n_{1}},\bar{v}_{n_{2}},w_{n_{3}})\|_{L^{2}(\R)}&\lesssim\frac{\|v_{n_{1}}\|_{2}\|v_{n_{2}}\|_{2}|w_{n_{3}}|}{|n-n_{1}||n-n_{3}|}\, ,\\
\intertext{and}
\|R^{1,t}_{V,n}(v_{n_{1}},\bar{w}_{n_{2}},v_{n_{3}})\|_{L^{2}(\R)}&\lesssim\frac{\|v_{n_{1}}\|_{2}|w_{n_{2}}|\|v_{n_{3}}\|_{2}}{|n-n_{1}||n-n_{3}|}.
\end{align*}
where the implicit constants do not depend on $n, n_1, n_2, n_3$.  
\end{lemma}
\begin{proof}
As mentioned in Remark \ref{asbefore} the operator $R^{1,t}_{I,n}$ was estimated in \cite{CHKP} and \cite{NP}. 

For $R^{1,t}_{II,n}, R^{1,t}_{III,n}$ the estimate is obvious since 
$\xi\in\mbox{supp}(\sigma_{n})$, otherwise the integrand is zero. 

For $R^{1,t}_{IV,n}, R^{1,t}_{V,n}$ it suffices to estimate only $R^{1,t}_{IV,n}$ since for $R^{1,t}_{V,n}$ similar considerations apply. To bound $\|R^{1,t}_{IV,n}(v_{n_{1}},\bar{v}_{n_{2}},w_{n_{3}})\|_{L^{2}(\R)}$ let $g\in L^{2}(\R), I_{n_{1}}=\mbox{supp}(\hat{v}_{n_{1}}), I_{n_{2}}=\mbox{supp}(\hat{\bar{v}}_{n_{2}})$, and consider the duality pairing
\begin{multline*}
\langle g,   R^{1,t}_{IV,n}(v_{n_{1}},\bar{v}_{n_{2}},w_{n_{3}})\rangle_{L^{2}(\R)} =  \Big|\int_{\R^{2}}\hat{g}(\xi)\ \sigma_{n}(\xi)\ \frac{w_{n_{3}}\hat{v}_{n_{1}}(\xi_{1})\hat{\bar{v}}_{n_{2}}(\xi-\xi_{1}-n_{3})}{(\xi-n_{3})(\xi-\xi_{1})}\ d\xi_{1} d\xi\Big| \\
= |w_{n_{3}}|\Big|\int_{\R^2}\hat{g}(\xi_{1}+\eta+n_{3})\ \sigma_{n}(\xi_{1}+\eta+n_{3})\ \frac{\hat{v}_{n_{1}}(\xi_{1})\hat{\bar{v}}_{n_{2}}(\eta)}{(\eta+\xi_{1})(\eta+n_{3})}\ d\xi_{1} d\eta\Big| \\
 \lesssim
\frac{\|\sigma_{n}\|_{\infty}|w_{n_{3}}|}{|n-n_{1}||n-n_{3}|}\int_{I_{n_{1}}}\int_{I_{n_{2}}}|\hat{g}(\xi_{1}+\eta+n_{3})||\hat{v}_{n_{1}}(\xi_{1})||\hat{\bar{v}}_{n_{2}}(\eta)|\ d\xi_{1} d\eta \\
\lesssim
\frac{|w_{n_{3}}|\|v_{n_{1}}\|_{2}\|v_{n_{2}}\|_{2}}{|n-n_{1}||n-n_{3}|}\ \|g\|_{2}\ |I_{n_{1}}|^{\frac12},
\end{multline*}
where we used that $\xi_{1}\in I_{n_{1}}, -\eta\in I_{n_{2}}, \xi\in\mbox{supp}(\sigma_{n})$ and H\"older's inequality. 
\end{proof}

\begin{remark}
\label{expl}
Notice that the same proof implies the following bounds
$$\|Q_{I,n}^{1,t}(v_{n_{1}},\bar{v}_{n_{2}},v_{n_{3}})\|_{L^{2}(\R)}\lesssim\|v_{n_{1}}\|_{2}\|v_{n_{2}}\|_{2}\|v_{n_{3}}\|_{2}$$
$$\|Q_{II,n}^{1,t}(w_{n_{1}},\bar{w}_{n_{2}},v_{n_{3}})\|_{L^{2}(\R)}\lesssim|w_{n_{1}}||w_{n_{2}}|\|v_{n_{3}}\|_{2}$$
$$\|Q_{III,n}^{1,t}(w_{n_{1}},\bar{v}_{n_{2}},w_{n_{3}})\|_{L^{2}(\R)}\lesssim|w_{n_{1}}|\|v_{n_{2}}\|_{2}|w_{n_{3}}|$$
$$\|Q_{IV,n}^{1,t}(v_{n_{1}},\bar{v}_{n_{2}},w_{n_{3}})\|_{L^{2}(\R)}\lesssim\|v_{n_{1}}\|_{2}\|v_{n_{2}}\|_{2}|w_{n_{3}}|$$
$$\|Q_{V,n}^{1,t}(v_{n_{1}},\bar{w}_{n_{2}},v_{n_{3}})\|_{L^{2}(\R)}\lesssim\|v_{n_{1}}\|_{2}|w_{n_{2}}|\|v_{n_{3}}\|_{2},$$
which will be used later in Lemmata \ref{finaal2} and \ref{dankda}. 
\end{remark}
For the $N_{21}^{t}$ operator the following bound holds 

\begin{lemma}
\label{fir1}
$$\|N_{21}^{t}(v)\|_{l^{2}(\Z)L^{2}(\R)}\lesssim N^{-\frac12+}(\|v\|^{3}_{L^{2}(\R)}+\|w\|_{L^{2}(\T)}\|v\|^{2}_{L^{2}(\R)}+\|w\|^{2}_{L^{2}(\T)}\|v\|_{L^{2}(\R)})$$
and
$$\|N_{21}^{t}(v)-N_{21}^{t}(u)\|_{l^{2}(\Z)L^{2}(\R)}\lesssim$$
$$N^{-\frac12+}(\|v\|^{2}_{L^{2}(\R)}+\|u\|^{2}_{L^{2}(\R)}+\|w\|_{L^{2}(\T)}(\|v\|_{L^{2}(\R)}+\|u\|_{L^{2}(\R)})+\|w\|^{2}_{L^{2}(\T)})\|v-u\|_{L^{2}(\R)}.$$
\end{lemma}
\begin{proof}
The sum
$$\sum_{A_{N}(n)^{c}}\tilde{Q}^{1,t}_{I,n}$$
was estimated in \cite{NP} Lemma $14$ giving an upper bound of 
the form $N^{-\frac12+}\|v\|_{2}^{3}$. 

For the sum that contains $\tilde{Q}^{1,t}_{II,n}, \tilde{Q}^{1,t}_{III,n}$ it suffices to estimate
$$\sum_{A_{N}(n)^{c}}\Big\|\tilde{Q}^{1,t}_{II,n}(w_{n_{1}},\bar{w}_{n_{2}},v_{n_{3}})\Big\|_{L^{2}(\R)}=\sum_{A_{N}(n)^{c}}\Big\|R^{1,t}_{II,n}(W_{n_{1}},\bar{W}_{n_{2}},V_{n_{3}})\Big\|_{L^{2}(\R)},$$
where we used \eqref{rela2}, \eqref{Sch} and \eqref{namee}. By Lemma \ref{fir} and H\"older's inequality we obtain the upper bound 
$$\sum_{A_{N}(n)^{c}}\frac{|W_{n_{1}}||W_{n_{2}}\|V_{n_{3}}\|_{2}}{|n-n_{1}||n-n_{3}|}\leq\Big(\sum_{A_{N}(n)^{c}}\frac1{|n-n_{1}|^{2}|n-n_{3}|^{2}}\Big)^{\frac12}\Big(\sum_{A_{N}(n)^{c}}|W_{n_{1}}|^{2}|W_{n_{2}}|^{2}\|V_{n_{3}}\|_{2}^{2}\Big)^{\frac12}.$$
The first sum is estimated by $N^{-\frac12+}$ with the use of \eqref{num} and then by taking the $l^{2}(\Z)$ norm and applying Young's inequality in $l^{1}(\Z)$ we arrive at
$$N^{-\frac12+}\Big(\sum_{n\in\Z}\sum_{A_{N}(n)^{c}}|W_{n_{1}}|^{2}|W_{n_{2}}|^{2}\|V_{n_{3}}\|_{2}^{2}\Big)^{\frac12}\leq N^{-\frac12+}\|W\|_{L^{2}(\T)}^{2}\|V\|_{L^{2}(\R)}=$$
$$N^{-\frac12+}\|w\|_{L^{2}(\T)}^{2}\|v\|_{L^{2}(\R)},$$
where we also used \eqref{Sch}. 

For the sum that contains $\tilde{Q}^{1,t}_{IV,n}$ and $\tilde{Q}^{1,t}_{V,n}$ we use again Lemma \ref{fir} and a similar argument as above, we leave the details to the reader . 
\end{proof}
In order to use a similar strategy to bound  the operator 
$N_{22}^{t}$,  the last term in equation \eqref{nex}, we need to 
use equation \eqref{nperpe} for the terms where 
$\partial_{t}(w_{n})$ appears and \eqref{mainmain} for the terms where 
$\partial_{t}(v_{n})$ appears. Because of the nonlinearity 
$G(w,v)$ there will be $51$ new operators in total. For example, the summand
$$\sum_{A_{N}(n)^{c}}\sigma_{n}(\xi)\int_{\R^2}\frac{e^{-2it(\xi-\xi_{1})(\xi-\xi_{3})}}{(\xi-\xi_{1})(\xi-\xi_{3})}\ \partial_{t}(\hat{v}_{n_{1}}(\xi_{1}))\hat{\bar{v}}_{n_{2}}(\xi-\xi_{1}-\xi_{3})\hat{v}_{n_{3}}(\xi_{3})\ d\xi_{1} d\xi_{3}$$
equals
\begin{multline*}
\sum_{A_{N}(n)^{c}}\sigma_{n}(\xi)\int_{\R^2}\frac{e^{-2it(\xi-\xi_{1})(\xi-\xi_{3})}}{(\xi-\xi_{1})(\xi-\xi_{3})}\Big(\mathcal F(R_{2}^{t}(v)(n_{1}))(\xi_{1})-\mathcal F(R_{1}^{t}(v)(n_{1}))(\xi_{1})\Big)\hat{\bar{v}}_{n_{2}}(\xi-\xi_{1}-\xi_{3})\hat{v}_{n_{3}}(\xi_{3})\ d\xi_{1} d\xi_{3} \\
+\sum_{A_{N}(n)^{c}}\sigma_{n}(\xi)\int_{\R^2}\frac{e^{-2it(\xi-\xi_{1})(\xi-\xi_{3})}}{(\xi-\xi_{1})(\xi-\xi_{3})}\mathcal F(N_{1}^{t}(v)(n_{1}))(\xi_{1})\hat{\bar{v}}_{n_{2}}(\xi-\xi_{1}-\xi_{3})\hat{v}_{n_{3}}(\xi_{3})\ d\xi_{1} d\xi_{3}
\end{multline*}
the summand 
$$\sum_{A_{N}(n)^{c}}\sigma_{n}(\xi)\ \frac{e^{-2it(\xi-n_{1})(n_{1}-n_{2})}}{(\xi-n_{1})(n_{1}-n_{2})}\ \partial_{t}(w_{n_{1}})\bar{w}_{n_{2}}\hat{v}_{n_{3}}(\xi-n_{1}+n_{2})$$
equals
\begin{multline*}
\sum_{A_{N}(n)^{c}}\sigma_{n}(\xi)\ \frac{e^{-2it(\xi-n_{1})(n_{1}-n_{2})}}{(\xi-n_{1})(n_{1}-n_{2})}\Big(\mathscr R_{2}^{t}(w)(n_{1})-\mathscr R_{1}^{t}(w)(n_{1})\Big)\bar{w}_{n_{2}}\hat{v}_{n_{3}}(\xi-n_{1}+n_{2}) \\
+\sum_{A_{N}(n)^{c}}\sigma_{n}(\xi)\ \frac{e^{-2it(\xi-n_{1})(n_{1}-n_{2})}}{(\xi-n_{1})(n_{1}-n_{2})}\ \mathscr N_{1}^{t}(w)(n_{1})\bar{w}_{n_{2}}\hat{v}_{n_{3}}(\xi-n_{1}+n_{2})\, ,
\end{multline*}
and the summand
$$\sum_{A_{N}(n)^{c}}\sigma_{n}(\xi)\ \frac{e^{-2it(\xi-n_{1})(n_{1}-n_{2})}}{(\xi-n_{1})(n_{1}-n_{2})}\ w_{n_{1}}\bar{w}_{n_{2}}\partial_{t}(\hat{v}_{n_{3}}(\xi-n_{1}+n_{2}))$$
equals
\begin{multline*}
\sum_{A_{N}(n)^{c}}\sigma_{n}(\xi)\frac{e^{-2it(\xi-n_{1})(n_{1}-n_{2})}}{(\xi-n_{1})(n_{1}-n_{2})}w_{n_{1}}\bar{w}_{n_{2}}\Big(\mathcal F(R_{2}^{t}(v)(n_{3}))(\xi-n_{1}+n_{2})-\mathcal F(R_{1}^{t}(v)(n_{3}))(\xi-n_{1}+n_{2})\Big) \\
+\sum_{A_{N}(n)^{c}}\sigma_{n}(\xi)\ \frac{e^{-2it(\xi-n_{1})(n_{1}-n_{2})}}{(\xi-n_{1})(n_{1}-n_{2})}\ w_{n_{1}}\bar{w}_{n_{2}}\mathcal F(N_{1}^{t}(v)(n_{3}))(\xi-n_{1}+n_{2})\, .
\end{multline*}
All summands that contain the resonant operators $\mathscr R_{2}^{t}(w), \mathscr R_{1}^{t}(w), \mathcal F(R_{2}^{t}(v)), \mathcal F(R_{1}^{t}(v))$ are good in the sense that they are controllable and all summands that contain the non-resonant operators $\mathscr N_{1}^{t}(w), \mathcal F(N_{1}^{t}(v))$ need to be decomposed further into "small" frequencies which give good operators and "big" frequencies using differentiation by parts. 

In order to be able to consistently write all these summands in a 
closed form we need the tree notation similarly as it was introduced in \cite{GKO}, but with some modifications.

\end{section}

\begin{section}{colored trees and the infinite iteration process}
\label{treenotation}
A tree $T$ is a finite, partially ordered set with the following properties:

\begin{itemize}
\item For any $a_{1}, a_{2}, a_{3}, a_{4}\in T$ if $a_{4}\leq a_{2}\leq a_{1}$ and $a_{4}\leq a_{3}\leq a_{1}$ then $a_{2}\leq a_{3}$ or $a_{3}\leq a_{2}$. 
\item There exists a maximum element $r\in T$, that is $a\leq r$ for all $a\in T$ which is called the \textbf{root}. 
\end{itemize}
We call the elements of $T$ the \textbf{nodes} of the tree and in this content we will say that $b\in T$ is a \textbf{child} of $a\in T$ (or equivalently, that $a$ is the \textbf{parent} of $b$) if $b\leq a, b\neq a$ and for all $c\in T$ such that $b\leq c\leq a$ we have either $b=c$ or $c=a$. 

A node $a\in T$ is called \textbf{terminal} if it has no children. A \textbf{nonterminal} node $a\in T$ is a node with exactly $3$ children $a_{1},$ the left child, $a_{2},$ the middle child, and $a_{3},$ the right child. We define the sets
\begin{equation}
\label{setsetset}
T^{0}=\{\mbox{all nonterminal nodes}\},
\end{equation}
and
\begin{equation}
\label{setsetset1}
T^{\infty}=\{\mbox{all terminal nodes}\}.
\end{equation}
Obviously, $T=T^{0}\cup T^{\infty}$, $T^{0}\cap T^{\infty}=\emptyset$ and if $|T^{0}|=j\in\Z_{+}$ we have $|T|=3j+1$ and $|T^{\infty}|=2j+1$. We denote the collection of trees with $j$ parental nodes by
\begin{equation}
\label{setsetset2}
T(j)=\{T \ \mbox{is a tree with}\ |T|=3j+1\}.
\end{equation}

So far, the notation agrees with the tree notation from \cite{GKO}. In addition, we color the trees by assigning a specific color, \textbf{black} or \textbf{red}, to each one of the nodes of such a tree. Let us describe the procedure: The first generation of \textbf{colored} trees, $C(1),$ consists of the following $5$ trees 
\begin{forest}
 [b[b][b][b]]
\end{forest} \quad 
\begin{forest}
[b[r][r][b]]
\end{forest} \quad
\begin{forest}
[b[r][b][r]]
\end{forest} \quad
\begin{forest}
[b[b][b][r]]
\end{forest} \quad
\begin{forest}
[b[b][r][b]]
\end{forest}\newline
These trees describe all possible "patterns" of the non-linearity $G(w,v)$, namely all combinations of 
$|v|^{2}v, |w|^{2}v, w^{2}\bar{v}, |v|^{2}w, v^{2}\bar{w}$ where 
$v$ is black and $w$ is red. There is also the red tree, 
which is not considered to belong to any generation, that 
plays an important role in the construction of the next 
generations and is simply given by 

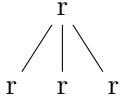
\begin{forest}
[r[r][r][r]]
\end{forest}\newline
Next we assume that the $J$th generation of colored trees, say 
$C(J),$ has been constructed, and we describe how the new generation 
$C(J+1)$ arises. Thus, let $T_{k}^{J}$ be one of the trees of the 
$C(J)$ family. We look at each of the $2J+1$ terminal nodes of 
$T_{k}^{J}$: 

\begin{itemize}
\item If one of these nodes is red then it gives rise to one new tree where this red node gave birth to three new red nodes. In other words, if a terminal node is red then attach the red tree to the tree $T_{k}^{J}$ at the red node. 
 \item If one of these nodes is black then it gives rise to five new trees where each one of them is born by attaching one of the trees of the first generation to the tree $T_{k}^{J}$ at the black node. 
\end{itemize}
We will denote by 
\begin{equation}
\label{howmany}
N(J)\coloneqq\mbox{card}(C(J)).
\end{equation}
Moreover, for a tree $T=T_{k}^{J}\in C(J)$ let
\begin{equation}
\label{blackandred}
b_{k}^{J}= \mbox{number of black terminal nodes of}\ T_{k}^{J},\ r_{k}^{J}= \mbox{number of red terminal nodes of}\ T_{k}^{J}
\end{equation}
and denote by 
\begin{equation}
\label{bigBR}
B_{k}^{J}=\{a\in T^{\infty}: a\ \mbox{is black}\}, R_{k}^{J}=\{a\in T^{\infty}:a\ \mbox{is red}\}. 
\end{equation}
Obviously we have the relations $B_{k}^{J}\cup R_{k}^{J}=T^{\infty}, B_{k}^{J}\cap R_{k}^{J}=\emptyset, \mbox{card}(B_{k}^{J})=b_{k}^{J}, \mbox{card}(R_{k}^{J})=r_{k}^{J}$ and 
\begin{equation}
\label{properttt}
b_{k}^{J}+r_{k}^{J}=2J+1,\ \max_{1\leq k\leq N(J)} b_{k}^{J}=2J+1,\ \mbox{and}\ \max_{1\leq k\leq N(J)} r_{k}^{J}=2J.
\end{equation}
The last two are true because there is at least one tree $T_{1}^{J}$ that consists of only black nodes. Therefore, for such tree we have $b_{1}^{J}=2J+1, r_{1}^{J}=0$, and there is also at least one tree $T_{2}^{J}$ with only one black terminal node, which implies $b_{2}^{J}=1, r_{2}^{J}=2J$. Also observe that by our construction there is no tree with only red terminal nodes. 

We also define the quantities 
\begin{equation}
\label{allblackred}
b_{J}=\sum_{k=1}^{N(J)}b_{k}^{J},\ r_{J}=\sum_{k=1}^{N(J)}r_{k}^{J},
\end{equation}
which respectively give the total number of black and red terminal nodes of the colored family $C(J)$. Notice that the number of colored trees of the next generation $C(J+1)$ is given by the formula
\begin{equation}
\label{counttrees}
N(J+1)=5b_{J}+r_{J}\, .
\end{equation}
This is because each one of the black nodes gives rise to $5$ new trees and each one of the red nodes gives rise to just $1$ new tree. 

Knowing the numbers $b_{k}^{J}, r_{k}^{J}$ for each tree $T_{k}^{J}\in C(J), 1\leq k\leq N(J)$ allows us to calculate the precise numbers $b_{J+1}$ and $r_{J+1}$ of the next generation by using the formulas
\begin{equation}
\label{blacknext}
b_{J+1}=\sum_{k=1}^{N(J)}\Big((5b_{k}^{J}+4)b_{k}^{J}+r_{k}^{J}b_{k}^{J}\Big)
\end{equation}
\begin{equation}
\label{rednext}
r_{J+1}=6b_{J}+2r_{J}+\sum_{k=1}^{N(J)}\Big(5b_{k}^{J}r_{k}^{J}+(r_{k}^{J})^{2}\Big).
\end{equation}
Indeed, each $b_{k}^{J}$ gives rise to $9+5(b_{k}^{J}-1)$ new black nodes and each red node $r_{k}^{J}$ leaves the number of black nodes the same as before. Also, each black node $b_{k}^{J}$ gives rise to $6+5r_{k}^{J}$ new red nodes and each red node $r_{k}^{J}$ gives rise to $3+r_{k}^{J}-1$ new red nodes. 

For our calculations it is important to know how fast the number $N(J)$ grows as $J$ approaches infinity. Since we have to count trees, 
 one expects a factorial growth and coloring the trees does not change this significantly: 
\begin{lemma}
\label{hopcor}
For every $J\in\mathbb N$ 
$$N(J)\leq\frac{10^{J}\ \Gamma(J+\frac12)}{\sqrt{\pi}},$$
where $\Gamma$ is the Gamma function. 
\end{lemma}
\begin{proof}
By \eqref{counttrees} and \eqref{properttt} we obtain 
$$N(J+1)=4b_{J}+N(J)(2J+1)\leq4(2J+1)N(J)+N(J)(2J+1)=5(2J+1)N(J).$$
Let us define a sequence $A(J)$ by the  recurrence relation $A(J+1)=5(2J+1)A(J), A(1)=5$. This can be solved explicitly in terms of the Gamma function using the equality $\Gamma(x+1)=x\Gamma(x), x>0$ and gives the result
$$A(J)=\frac{10^{J}\ \Gamma(J+\frac12)}{\sqrt{\pi}}.$$
An easy induction argument shows that for all $J\in\mathbb N$ we have $N(J)\leq A(J)$ which finishes the proof.
\end{proof}
Using the equality
\begin{equation}
\label{gammaa}
\Gamma(J+\frac12)=\frac{(2J-1)!!}{2^{J}}\ \sqrt{\pi},
\end{equation}
$J\in\mathbb N,$ where the double factorial $(2J-1)!!=1\cdot3\cdot5\cdot\ldots\cdot(2J-1),$ we obtain the bound
\begin{equation}
\label{growthoftrees}
N(J)\leq 5^{J}(2J-1)!!,
\end{equation} 
for all $J\in\mathbb N$.

Given a colored tree $T=T_{k}^{J}$ of the $C(J)$ family we define an index function $n:T_{k}^{J}\to\Z$ such that
\begin{itemize}
\item If $a$ is a black node in $T^{0}$ then $n_{a}\approx n_{a_{1}}-n_{a_{2}}+n_{a_{3}}$ (see \eqref{approxim}) where $a_{1}, a_{2}, a_{3}$ are the children of $a$.
\item If $a$ is a red node in $T^{0}$ then $n_{a}=n_{a_{1}}-n_{a_{2}}+n_{a_{3}},$ where $a_{1}, a_{2}, a_{3}$ are the children of $a$,
\item $n_{a}\not\approx n_{a_{1}}$ and $n_{a}\not\approx n_{a_{3}}$ for all black nodes $a\in T^{0}$ and $n_{a}\neq n_{a_{1}}$ and $n_{a}\neq n_{a_{3}}$ for all red nodes $a\in T^{0}$.
\item $|\mu_{1}|\coloneqq 2|n_{r}-n_{r_{1}}||n_{r}-n_{r_{3}}|>N$, where $r$ is the root of $T_{k}^{J}$.
\end{itemize}
We denote the collection of all such index functions by $\mathcal R(T_{k}^{J})$.

Similar to what was done in \cite{GKO}, given a colored tree $T$ 
in $C(J)$ and an index function $n\in\mathcal R(T),$ we need to 
keep track of the generations of frequencies. Consider the very 
first tree $T_{1},$ that is, the root $r$ and its children 
$r_{1}, r_{2}, r_{3}$. We define the first generation of 
frequencies by 
$$(n^{(1)},n_{1}^{(1)},n_{2}^{(1)},n_{3}^{(1)})\coloneqq (n_{r},n_{r_{1}},n_{r_{2}},n_{r_{3}}).$$
From the definition of the index function we have
$$n^{(1)}\approx n_{1}^{(1)}-n_{2}^{(1)}+n_{3}^{(1)},\ n_{1}^{(1)}\not\approx n^{(1)}\not\approx n_{3}^{(1)},$$
since the root node is colored black. The tree $T_{2}$ of the second generation is obtained from $T_{1}$ by changing one of its terminal nodes $a=r_{k}\in T_{1}^{\infty}$ for some $k=1,2,3$ into a nonterminal node. Then, the second generation of frequencies is defined by
$$(n^{(2)},n_{1}^{(2)},n_{2}^{(2)},n_{3}^{(2)})\coloneqq(n_{a},n_{a_{1}},n_{a_{2}},n_{a_{3}}).$$
Thus we have $n^{(2)}=n_{k}^{(1)}$ for some $k=1,2,3$ and from the definition of the index function we get
$$n^{(2)}\approx n_{1}^{(2)}-n_{2}^{(2)}+n_{3}^{(2)},\ n_{1}^{(2)}\not\approx n^{(2)}\not\approx n_{3}^{(2)}$$
if $n_{k}^{(1)}$ is black or
$$n^{(2)}= n_{1}^{(2)}-n_{2}^{(2)}+n_{3}^{(2)},\ n_{1}^{(2)}\neq n^{(2)}\neq n_{3}^{(2)}$$
if $n_{k}^{(1)}$ is red. After $j-1$ steps, the tree $T_{j}$ of the $j$th generation is obtained from $T_{j-1}$ by changing one of its terminal nodes $a\in T_{j-1}^{\infty}$ into a nonterminal node. Then, the $j$th generation frequencies are defined as 
$$(n^{(j)},n_{1}^{(j)},n_{2}^{(j)},n_{3}^{(j)})\coloneqq(n_{a},n_{a_{1}},n_{a_{2}},n_{a_{3}})$$
and we have $n^{(j)}=n_{k}^{(m)}(=n_{a})$ for some $m=1,2,\ldots,j-1$ and $k=1,2,3,$ since this corresponds to the frequency of some terminal node in $T_{j-1}$. In addition, from the definition of the index function we have
$$n^{(j)}\approx n_{1}^{(j)}-n_{2}^{(j)}+n_{3}^{(j)},\ n_{1}^{(j)}\not\approx n^{(j)}\not\approx n_{3}^{(j)}$$
if $n_{k}^{(m)}$ is black or
$$n^{(j)}= n_{1}^{(j)}-n_{2}^{(j)}+n_{3}^{(j)},\ n_{1}^{(j)}\neq n^{(j)}\neq n_{3}^{(j)}$$
if $n_{k}^{(m)}$ is red. 

We use $\mu_{j}$ to denote the corresponding phase factor introduced at the $j$th generation. That is,
\begin{equation}
\label{muuu}
\mu_{j}=2(n^{(j)}-n_{1}^{(j)})(n^{(j)}-n_{3}^{(j)}),
\end{equation}
and we also introduce the quantities
\begin{equation}
\label{qqq}
\tilde\mu_{J}=\sum_{j=1}^{J}\mu_{j},\ \hat{\mu}_{J}=\prod_{j=1}^{J}\tilde\mu_{j}.
\end{equation}
We should keep in mind that every time we apply differentiation 
by parts and split the operators, we need to control the new 
frequencies that arise from this procedure. For this reason, 
we need to define the sets
\begin{equation}
\label{sesee}
C_{J}\coloneqq\{|\tilde\mu_{J+1}|\leq(2J+3)^{3}|\tilde\mu_{J}|^{1-\frac1{100}}\}\cup\{|\tilde\mu_{J+1}|\leq(2J+3)^{3}|\mu_{1}|^{1-\frac1{100}}\}.
\end{equation}

Let us denote by $T_{\alpha}$ all the nodes of the tree $T$ that are descendants of the node $\alpha\in T^{0}$, i.e. $T_{\alpha}=\{\beta\in T:\beta\leq\alpha,\ \beta\neq\alpha\}$. 

We also need to define the \textbf{principal and final "signs" of a node} $a\in T$ which are functions from the tree $T$ into the set $\{\pm1\}$:
\begin{equation}
\label{signsign}
\mbox{psgn}(a)=\begin{cases}
+1,\ a\ \mbox{is not the middle child of his parent}\\
+1,\ a=r,\ \mbox{the root node}\\
-1,\ a\ \mbox{is the middle child of his parent}
\end{cases}
\end{equation}
\begin{equation}
\label{signsignsign}
\mbox{fsgn}(a)=\begin{cases}
+1,\ \mbox{psgn}(a)=+1\ \mbox{and}\ a\ \mbox{has an even number of middle predecessors}\\
-1,\ \mbox{psgn}(a)=+1\ \mbox{and}\ a\ \mbox{has an odd number of middle predecessors}\\
-1,\ \mbox{psgn}(a)=-1\ \mbox{and}\ a\ \mbox{has an even number of middle predecessors}\\
+1,\ \mbox{psgn}(a)=-1\ \mbox{and}\ a\ \mbox{has an odd number of middle predecessors},
\end{cases}
\end{equation}
where the root node $r\in T$ is not considered a middle parent. 

Next we define two "prototype" operators in the following way. Suppose that $T\in T(J)$ (see \eqref{setsetset2}) is a tree of only black nodes. Let $\tilde q^{J,t}_{T,\mathbf n}$ and $R^{J,t}_{T,\mathbf n}$ be related as
\begin{equation}
\label{oops}
\mathcal F(\tilde q^{J,t}_{T,\mathbf n}(\{v_{n_\beta}\}_{\beta\in T^{\infty}}))(\xi)=e^{-it\xi^{2}}\mathcal F(R^{J,t}_{T,\mathbf n}(\{e^{-it\partial_{x}^{2}}v_{n_\beta}\}_{\beta\in T^{\infty}}))(\xi),
\end{equation}
where the operator $R^{J,t}_{T,\mathbf n}$ acts on the functions $\{v_{n_\beta}\}_{\beta\in T^{\infty}}$ as
\begin{equation}
\label{oops1}
R^{J,t}_{T,\mathbf n}(\{v_{n_\beta}\}_{\beta\in T^{\infty}})(x)=\int_{\R^{2J+1}}K^{(J)}_{T}(x,\{x_{\beta}\}_{\beta\in T^{\infty}})\Big[\otimes_{\beta\in T^{\infty}}v_{n_\beta}(x_{\beta})\Big]\ \prod_{\beta\in T^{\infty}} dx_{\beta},
\end{equation}
and the Kernel $K^{(J)}_{T,\mathbf n}$ is defined as 
\begin{equation}
\label{oopss}
K^{(J)}_{T,\mathbf n}(x,\{x_{\beta}\}_{\beta\in T^{\infty}})=\mathcal F^{-1}(\rho^{(J)}_{T,\mathbf n})(\{x-x_{\beta}\}_{\beta\in T^{\infty}}),
\end{equation}
where the formula for the function $\rho^{(J)}_{T,\mathbf n}$ with ($|T^{\infty}|=2J+1$)-variables, $\xi_{\beta}$, $\beta\in T^{\infty}$ is
\begin{equation}
\label{jc}
\rho^{(J)}_{T,\mathbf n}(\{\xi_{\beta}\}_{\beta\in T^{\infty}})=\Big[\prod_{\alpha\in T^0}\sigma_{n_{\alpha}}\Big(\sum_{\beta\in T^{\infty}\cap T_{\alpha}}\mbox{fsgn}(\beta)\ \xi_{\beta}\Big)\Big]\frac1{\hat{\mu}_{T}}.
\end{equation}
We denote by 
\begin{equation}
\label{yeah}
\hat{\mu}_{T}=\prod_{\alpha\in T^0}\tilde\mu_{\alpha},\ \tilde\mu_{\alpha}=\sum_{\beta\in T^{0}\setminus T_{\alpha}}\mu_{\beta},
\end{equation}
and for $\beta\in T^{0}$ we have
\begin{equation}
\label{yyeah}
\mu_{\beta}=2(\xi_{\beta}-\xi_{\beta_{1}})(\xi_{\beta}-\xi_{\beta_{3}}),
\end{equation}
where we impose the relation $\xi_{\alpha}=\xi_{\alpha_{1}}-\xi_{\alpha_{2}}+\xi_{\alpha_{3}}$ for every $\alpha\in T^{0}$ that appears in the calculations, until we reach the terminal nodes of $T^{\infty}$. This is due to the fact that  in the definition of the function $ \rho^{J,t}_{T}$ we need the variables "$\xi$" to be assigned only at the terminal nodes of the tree $T$. 
We use the notation $\mu_{\beta}$ similarly to $\mu_{j}$ of equation \eqref{muuu}, because this is the ``continuous" version of the discrete case. In addition, the variables 
$\xi_{\alpha_{1}}, \xi_{\alpha_{2}}, \xi_{\alpha_{3}}$ that appear in expression \eqref{jc} are supported in such a way that 
$\xi_{\alpha_{1}}\approx n_{\alpha_{1}}, \xi_{\alpha_{2}}\approx n_{\alpha_{2}}, \xi_{\alpha_{3}}\approx n_{\alpha_{3}}$,  
due to the support properties of the cut--off functions 
$\sigma_{n_{\alpha}}$. 
Therefore, $|\hat{\mu}_{T}|\sim|\hat{\mu}_{J}|$.

Notice that if $\{\beta_{1},\ldots,\beta_{2J+1}\}=T^{\infty}$, 
then we may rewrite \eqref{oops1} as 
\begin{equation}
\label{newoops1}
R^{J,t}_{T,\mathbf n}(v_{n_{\beta_{1}}},\ldots,v_{n_{\beta_{2J+1}}})(x)=
\end{equation}
$$\int_{\R}e^{ix\xi}\Big(\int_{\R^{2J}}\rho^{(J)}_{T,\mathbf n}(\xi_{\beta_{1}},\ldots,\xi_{\beta_{2J}},\xi-\sum_{k=1}^{2J}\xi_{\beta_{k}})\prod_{k=1}^{2J}\hat{v}_{n_{\beta_{k}}}(\xi_{\beta_{k}})\ \hat{v}_{n_{\beta_{2J+1}}}(\xi-\sum_{k=1}^{2J}\xi_{\beta_{k}})\prod_{k=1}^{2J}d\xi_{\beta_{k}}\Big)d\xi$$
which implies 
\begin{equation}
\label{fourieroops1}
\mathcal F(R^{J,t}_{T,\mathbf n}(v_{n_{\beta_{1}}},\ldots,v_{n_{\beta_{2J+1}}}))(\xi)=
\end{equation}
$$\int_{\R^{2J}}\rho^{(J)}_{T,\mathbf n}(\xi_{\beta_{1}},\ldots,\xi_{\beta_{2J}},\xi-\sum_{k=1}^{2J}\xi_{\beta_{k}})\prod_{k=1}^{2J}\hat{v}_{n_{\beta_{k}}}(\xi_{\beta_{k}})\ \hat{v}_{n_{\beta_{2J+1}}}(\xi-\sum_{k=1}^{2J}\xi_{\beta_{k}})\prod_{k=1}^{2J}d\xi_{\beta_{k}}.$$
Such an operator was studied in \cite{CHKP} Lemma $21$ and in \cite{NP} Lemma $21$. 

Our goal is to define the operators $\tilde q^{J,t}_{T,\mathbf n}$ and $R^{J,t}_{T,\mathbf n}$ for any colored tree $T_{k}^{J}$ of the $C(J)$ family. From \eqref{bigBR} we know that $B_{k}^{J}\cup R_{k}^{J}=(T_{k}^{J})^{\infty}$. If $R_{k}^{J}=\emptyset$ then the tree is black and the operators have already been defined by \eqref{fourieroops1}. Thus, assume $R_{k}^{J}=\{r_{m}\}_{m=1}^{r_{k}^{J}}\neq\emptyset, B_{k}^{J}=\{b_{m}\}_{m=1}^{b_{k}^{J}}$ and consider functions $\{v_{n_{b_{m}}}\}_{m=1}^{b_{k}^{J}}$ and Fourier coefficients $\{w_{n_{r_{\tilde{m}}}}\}_{\tilde{m}=1}^{r_{k}^{J}}$. Let $\{v_{n_{r_{\tilde{m}}}}\}$ be defined by 
\begin{equation}
\label{functionsalle}
\hat{v}_{n_{r_{\tilde{m}}}}(\xi)=w_{n_{r_{\tilde{m}}}}\delta_{n_{r_{\tilde{m}}}}(\xi),
\end{equation}
for all $\tilde{m}\in\{1,\ldots, r_{k}^{J}\}$. Then the operator $R^{J,t}_{T_{k}^{J}, \mathbf n}$ is defined as 
\begin{equation}
\label{coloredoperators1}
\mathcal F(R^{J,t}_{T_{k}^{J}, \mathbf n}(\{v_{n_{b_{m}}}\}_{m=1}^{b_{k}^{J}},\{w_{n_{r_{\tilde{m}}}}\}_{\tilde{m}=1}^{r_{k}^{J}}))(\xi)=
\end{equation}
$$\mathcal F(R^{J,t}_{T, \mathbf n}(\{v_{n_{b_{m}}}\}_{m=1}^{b_{k}^{J}},\{v_{n_{r_{\tilde{m}}}}\}_{\tilde{m}=1}^{r_{k}^{J}}))(\xi)$$
and 
\begin{equation}
\label{coloredoperators2}
\mathcal F(\tilde{q}^{J,t}_{T_{k}^{J}, \mathbf n}(\{v_{n_{b_{m}}}\}_{m=1}^{b_{k}^{J}},\{w_{n_{r_{\tilde{m}}}}\}_{\tilde{m}=1}^{r_{k}^{J}}))(\xi)=
\end{equation}
$$e^{-it\xi^{2}}\mathcal F(R^{J,t}_{T_{k}^{J},\mathbf n}(\{e^{-it\partial_{x}^{2}}v_{n_{b_{m}}}\}_{m=1}^{b_{k}^{J}},\{e^{itn^{2}_{r_{\tilde{m}}}}w_{n_{r_{\tilde{m}}}}\}_{\tilde{m}=1}^{r_{k}^{J}}))(\xi).$$
For these operators the following holds.

\begin{lemma}
\label{indu}
$$\|R^{J,t}_{T_{k}^{J}, \mathbf n}(\{v_{n_{b_{m}}}\}_{m=1}^{b_{k}^{J}},\{w_{n_{r_{\tilde{m}}}}\}_{\tilde{m}=1}^{r_{k}^{J}})\|_{L^{2}(\R)}\lesssim\frac{\prod_{m=1}^{b_{k}^{J}}\|v_{n_{b_{m}}}\|_{2}\ \prod_{\tilde{m}=1}^{r_{k}^{J}}|w_{n_{r_{\tilde{m}}}}|}{|\hat{\mu}_{T_{k}^{J}}|}$$
\end{lemma}
\begin{proof}
  The proof of the above bound is similar to 
  the strategy of the proof of Lemma \ref{fir}: a repeated use of duality and H\"older's inequality.  We leave the details to the reader. 	
\end{proof}

Next, given a colored tree $T=T_{k}^{J}$ of the $C(J)$ family and $\alpha\in T^{\infty}$ we define the operators $\mathbf R_{2}^{t,\alpha}-\mathbf R_{1}^{t\alpha}, \mathbf N_{1}^{t,\alpha}$ by
\begin{equation}
\label{newdef33}
\mathbf R_{2}^{t,\alpha}-\mathbf R_{1}^{t,\alpha}=\begin{cases} R_{2}^{t}-R_{1}^{t} &,\ \alpha\in B_{k}^{J}\\ \mathscr R_{2}^{t}-\mathscr R_{1}^{t} &,\ \alpha\in R_{k}^{J}\\ \end{cases}, \quad\mathbf N_{1}^{t,\alpha}\coloneqq\begin{cases} N_{1}^{t}&,\ \alpha\in B_{k}^{J}\\ \mathscr N_{1}^{t} &,\ \alpha\in R_{k}^{J}\\ \end{cases}.
\end{equation}
Next, for such a tree $T=T_{k}^{J},$ index function 
$\mathbf n\in\mathcal R(T),$ $\alpha\in T^{\infty}$ and set 
of functions 
$\{v_{n_{b_{m}}}\}_{m=1}^{b_{k}^{J}},\{w_{n_{r_{\tilde{m}}}}\}_{\tilde{m}=1}^{r_{k}^{J}}$ 
we define the action of the operator $\mathbf N_{1}^{t,\alpha}$ 
onto the set of functions to be the same set as before but with the difference that we have substituted the function 
$f_{n_{\alpha}}\coloneqq v_{n_{\alpha}}\chi_{B_{k}^{J}}(\alpha)+w_{n_{\alpha}}\chi_{R_{k}^{J}}(\alpha)$ with 
$\mathbf N_{1}^{t,\alpha}(f_{n_{\alpha}})$. 
Similarly, we define the action of the operator 
$\mathbf R_{2}^{t,\alpha}-\mathbf R_{1}^{t,\alpha}$ onto 
the set of functions 
$\{v_{n_{b_{m}}}\}_{m=1}^{b_{k}^{J}},\{w_{n_{r_{\tilde{m}}}}\}_{\tilde{m}=1}^{r_{k}^{J}}$.

The operator of the $J$th step, $J\geq 2$, that we want to estimate, 
is given by the formula
\begin{equation}
\label{fina}
N_{2}^{(J)}(v)(n)\coloneqq\sum_{T\in C(J-1)}\sum_{\alpha\in T^{\infty}}\sum_{\substack{\mathbf n\in\mathcal R(T)\\ \mathbf n_{r}=n}}\tilde q^{J-1,t}_{T,\mathbf n}(\mathbf N_{1}^{t,\alpha}(\{v_{n_{b_{m}}}\}_{m=1}^{b_{k}^{J}},\{w_{n_{r_{\tilde{m}}}}\}_{\tilde{m}=1}^{r_{k}^{J}})).
\end{equation}
Applying differentiation by parts on the Fourier side, keeping in mind that from the splitting procedure we are on the sets 
$A_{N}(n)^{c},C_{1}^{c},\ldots,C_{J-1}^{c}$, 
we obtain the expression
\begin{equation}
\label{fina1}
N_{2}^{(J)}(v)(n)=\partial_{t}(N_{0}^{(J+1)}(v)(n))+N_{r}^{(J+1)}(v)(n)+N^{(J+1)}(v)(n), 
\end{equation}
where
\begin{equation}
\label{fina2}
N_{0}^{(J+1)}(v)(n)\coloneqq\sum_{T\in C(J)}\sum_{\substack{\mathbf n\in\mathcal R(T)\\ \mathbf n_{r}=n}}\tilde q^{J,t}_{T,\mathbf n}(\{v_{n_{b_{m}}}\}_{m=1}^{b_{k}^{J}},\{w_{n_{r_{\tilde{m}}}}\}_{\tilde{m}=1}^{r_{k}^{J}}),
\end{equation}
and
\begin{equation}
\label{fina3}
N_{r}^{(J+1)}(v)(n)\coloneqq\sum_{T\in C(J)}\sum_{\alpha\in T^{\infty}}\sum_{\substack{\mathbf n\in\mathcal R(T)\\ \mathbf n_{r}=n}}\tilde q^{J,t}_{T,\mathbf n}((\mathbf R^{t,\alpha}_{2}-\mathbf R^{t,\alpha}_{1})(\{v_{n_{b_{m}}}\}_{m=1}^{b_{k}^{J}},\{w_{n_{r_{\tilde{m}}}}\}_{\tilde{m}=1}^{r_{k}^{J}})),
\end{equation}
and
\begin{equation}
\label{fina4}
N^{(J+1)}(v)(n)\coloneqq\sum_{T\in C(J)}\sum_{\alpha\in T^{\infty}}\sum_{\substack{\mathbf n\in\mathcal R(T)\\ \mathbf n_{r}=n}}\tilde q^{J,t}_{T,\mathbf n}(\mathbf N_{1}^{t,\alpha}(\{v_{n_{b_{m}}}\}_{m=1}^{b_{k}^{J}},\{w_{n_{r_{\tilde{m}}}}\}_{\tilde{m}=1}^{r_{k}^{J}})).
\end{equation}
We also split the operator $N^{(J+1)}$ as the sum
\begin{equation}
\label{fina5}
N^{(J+1)}=N_{1}^{(J+1)}+N_{2}^{(J+1)},
\end{equation}
where $N_{1}^{(J+1)}$ is the restriction of $N^{(J+1)}$ onto $C_{J}$ and $N_{2}^{(J+1)}$ onto $C_{J}^{c}$. 

First we estimate the operators $N_{0}^{(J+1)}$ and $N_{r}^{(J+1)}$.

\begin{lemma}
\label{finaal}
$$\|N_{0}^{(J+1)}(v)\|_{l^{2}(\Z)L^{2}(\R)}\lesssim N^{-\frac{J}{2}+\frac{(J-1)}{200}+}(\|v\|_{L^{2}(\R)}+\|w\|_{L^{2}(\T)})^{2J+1}$$
and 
$$\|N_{0}^{(J+1)}(v)-N_{0}^{(J+1)}(u)\|_{l^{2}(\Z)L^{2}(\R)}\lesssim N^{-\frac{J}{2}+\frac{(J-1)}{200}+}(\|v\|_{L^{2}(\R)}+\|u\|_{L^{2}(\R)}+\|w\|_{L^{2}(\T)})^{2J}\|v-u\|_{L^{2}(\R)}.$$

$$\|N_{r}^{(J+1)}(v)\|_{l^{2}(\Z)L^{2}(\R)}\lesssim N^{-\frac{J}{2}+\frac{(J-1)}{200}+}(\|v\|_{L^{2}(\R)}+\|w\|_{L^{2}(\T)})^{2J+3}$$
and
$$\|N_{r}^{(J+1)}(v)-N_{r}^{(J+1)}(u)\|_{l^{2}(\Z)L^{2}(\R)}\lesssim N^{-\frac{J}{2}+\frac{(J-1)}{200}+}(\|v\|_{L^{2}(\R)}+\|u\|_{L^{2}(\R)}+\|w\|_{L^{2}(\T)})^{2J+2}\|v-u\|_{L^{2}(\R)}.$$
\end{lemma}
\begin{proof}
By \eqref{num} for fixed $n^{(j)}$ and $\mu_{j}$ there are at most $o(|\mu_{j}|^{+})$ many choices for $n_{1}^{(j)},n_{2}^{(j)},n_{3}^{(j)}$. In addition, let us observe that $\mu_{j}$ is determined by $\tilde\mu_{1},\ldots,\tilde\mu_{j}$ and $|\mu_{j}|\lesssim\max(|\tilde\mu_{j-1}|,|\tilde\mu_{j}|)$, since $\mu_{j}=\tilde\mu_{j}-\tilde\mu_{j-1}$. Then, for a fixed tree $T=T_{k}^{J}\in C(J)$, by Lemma \ref{indu} the estimate for the operator $\tilde q^{J,t}_{T,\mathbf n}$ is as follows (remember that $|\hat{\mu}_{T}|\sim|\hat{\mu}_{J}|=\prod_{k=1}^{J}|\tilde\mu_{k}|$)
$$\sum_{\substack{\mathbf n\in\mathcal R(T)\\ \mathbf n_{r}=n}}\|\tilde q^{J,t}_{T,\mathbf n}(\{v_{n_{b_{m}}}\}_{m=1}^{b_{k}^{J}},\{w_{n_{r_{\tilde{m}}}}\}_{\tilde{m}=1}^{r_{k}^{J}})\|_{2}\lesssim$$
$$\sum_{\substack{\mathbf n\in\mathcal R(T)\\ \mathbf n_{r}=n}}\Big(\prod_{m=1}^{b_{k}^{J}}\|v_{n_{b_{m}}}\|_{2}\ \prod_{\tilde{m}=1}^{r_{k}^{J}}|w_{n_{r_{\tilde{m}}}}|\Big)\Big(\prod_{k=1}^{J}\frac1{|\tilde\mu_{k}|}\Big),$$
and, by H\"older's inequality, this is bounded from above by
\begin{equation}
\label{hahah}
\Big(\sum_{\substack{|\mu_{1}|>N\\ |\tilde\mu_{j}|>(2j+1)^{3}N^{1-\frac1{100}}\\ j=2,\ldots,J}}\prod_{k=1}^{J}\frac1{|\tilde\mu_{k}|^{2}}|\mu_{k}|^{+}\Big)^{\frac1{2}}\Big(\sum_{\substack{\mathbf n\in\mathcal R(T)\\ \mathbf n_{r}=n}}\prod_{m=1}^{b_{k}^{J}}\|v_{n_{b_{m}}}\|_{2}^{2}\ \prod_{\tilde{m}=1}^{r_{k}^{J}}|w_{n_{r_{\tilde{m}}}}|^{2}\Big)^{\frac1{q}}.
\end{equation}
The first sum behaves like $N^{-\frac{J}{2}+\frac{(J-1)}{200}+}$ and for the remaining part we take the $l^{2}(\Z)$ norm in $n$ and by the use of Young's inequality we obtain the upper bound of
$$N^{-\frac{J}{2}+\frac{(J-1)}{200}+}\|v\|_{L^{2}(\R)}^{b_{k}^{J}}\|w\|_{L^{2}(\T)}^{r_{k}^{J}}.$$
Collecting terms, one sees that this proves the bound for $\|N_{0}^{(J+1)}(v)\|_{l^{2}(\Z)L^{2}(\R)}$.

Note that there is an extra factor $\sim J$ when we estimate the differences $N_{0}^{(J+1)}(v)-N_{0}^{(J+1)}(w)$ since $|a^{2J+1}-b^{2J+1}|\lesssim(\sum_{j=1}^{2J+1}a^{2J+1-j}b^{j-1})|a-b|$ has $O(J)$ many terms. Also, we have $N(J)=\mbox{card}(C(J))$ many summands in the operator $N_{0}^{(J+1)}$ since there are $N(J)$ many trees of the $J$th generation. However, these observations do not cause any problem since the constant that we obtain from estimating the first sum of \eqref{hahah} decays like a fractional power of a double factorial in $J$, or to be more precise, with the use of \eqref{growthoftrees} we have the following behaviour in $J$
\begin{equation}t
\label{factori}
\frac{5^{J}\cdot (2J-1)!!}{(2J-1)!!^{\frac32}}=\frac{5^{J}}{(2J-1)!!^{\frac12}}.
\end{equation}
For the operator $N_{r}^{(J+1)}$ the proof is the same but in addition we use Lemma \ref{lem} and Remark \ref{outoftheblue} for the operator $\mathbf R_{2}^{t}-\mathbf R_{1}^{t}$.
\end{proof}
Then the estimate for the operator $N_{1}^{(J+1)}$ is the following.

\begin{lemma}
\label{finaal2}
$$\|N_{1}^{(J+1)}(v)\|_{l^{2}(\Z)L^{2}(\R)}\lesssim N^{-\frac{J-1}{2}+\frac{(J-2)}{200}+}(\|v\|_{L^{2}(\R)}+\|w\|_{L^{2}(\T)})^{2J+3}$$
and
$$\|N_{1}^{(J+1)}(v)-N_{1}^{(J+1)}(u)\|_{l^{2}(\Z)L^{2}(\R)}\lesssim N^{-\frac{J-1}{2}+\frac{(J-2)}{200}+}(\|v\|_{L^{2}(\R)}+\|u\|_{L^{2}(\R)}+\|w\|_{L^{2}(\T)})^{2J+2}\|v-u\|_{L^{2}(\R)}.$$
\end{lemma}
\begin{proof}
As before, for fixed $n^{(j)}$ and $\mu_{j}$ there are at most $o(|\mu_{j}|^{+})$ many choices for $n_{1}^{(1)}, n_{2}^{(1)}, n_{3}^{(1)}$ and note that $\mu_{j}$ is determined by $\tilde\mu_{1},\ldots,\tilde\mu_{j}$.

Let us assume that 
$|\tilde\mu_{J+1}|=|\tilde\mu_{J}+\mu_{J+1}|\lesssim(2J+3)^{3}|\tilde\mu_{J}|^{1-\frac1{100}}$ 
holds in \eqref{sesee}. Then, 
$|\mu_{J+1}|\lesssim|\tilde\mu_{J}|$ and for fixed 
$\tilde\mu_{J}$ there are at most 
$o(|\tilde\mu_{J}|^{1-\frac1{100}})$ many choices for 
$\tilde\mu_{J+1}$ and therefore, also for $\mu_{J+1}=\tilde\mu_{J+1}-\tilde\mu_{J}$. 
For a fixed tree 
$T=T_{k}^{J}\in C(J),\alpha\in B_{k}^{J}\subset T^{\infty}$, 
Lemma \ref{indu}, Remark \ref{expl} and the definition of the 
operator $N_{1}^{t}(v)$, see \eqref{main10}, we estimate 
$\tilde q^{J,t}_{T,\mathbf n}$ as follows (remember that 
$|\hat{\mu}_{T}|\sim|\hat{\mu}_{J}|=\prod_{k=1}^{J}|\tilde\mu_{k}|$)
$$\sum_{\substack{\mathbf n\in\mathcal R(T)\\ \mathbf n_{r}=n}}\|\tilde q^{J,t}_{T,\mathbf n}(\mathbf N_{1}^{t,\alpha}(\{v_{n_{b_{m}}}\}_{m=1}^{b_{k}^{J}},\{w_{n_{r_{\tilde{m}}}}\}_{\tilde{m}=1}^{r_{k}^{J}}))\|_{2}\lesssim$$
$$\sum_{\substack{\mathbf n\in\mathcal R(T)\\ \mathbf n_{r}=n}}\Big(\Big[\|v_{n_{\alpha_{1}}}\|_{2}\|v_{n_{\alpha_{2}}}\|_{2}\|v_{n_{\alpha_{3}}}\|_{2}+\|v_{n_{\alpha_{1}}}\|_{2}|w_{n_{\alpha_{2}}}|\|v_{n_{\alpha_{3}}}\|_{2}+\|v_{n_{\alpha_{1}}}\|_{2}\|v_{n_{\alpha_{2}}}\|_{2}|w_{n_{\alpha_{3}}}|+$$
$$|w_{n_{\alpha_{1}}}|\|v_{n_{\alpha_{2}}}\|_{2}|w_{n_{\alpha_{3}}}|+|w_{n_{\alpha_{1}}}\|_{2}|w_{n_{\alpha_{2}}}|\|v_{n_{\alpha_{3}}}\|_{2}\Big]\prod_{\beta\in B_{k}^{J}\setminus\{\alpha\}}\|v_{n_\beta}\|_{2}\ \prod_{\tilde{m}=1}^{r_{k}^{J}}|w_{n_{r_{\tilde{m}}}}|\Big)\Big(\prod_{k=1}^{J}\frac1{|\tilde\mu_{k}|}\Big).$$
Then for the $\|v_{n_{\alpha_{1}}}\|_{2}\|v_{n_{\alpha_{2}}}\|_{2}\|v_{n_{\alpha_{3}}}\|_{2}$ term, the same calculations work for the other terms, we apply H\"older's inequality and obtain the upper bound
\begin{equation}
\label{hahah2}
\Big(\sum_{\substack{|\mu_{1}|>N\\ |\tilde\mu_{j}|>(2j+1)^{3}N^{1-\frac1{100}}\\ j=2,\ldots,J}}|\tilde\mu_{J}|^{1-\frac1{100}+}\prod_{k=1}^{J}\frac1{|\tilde\mu_{k}|^{2}}|\mu_{k}|^{+}\Big)^{\frac1{2}}
\end{equation}
$$\Big(\sum_{\substack{\mathbf n\in\mathcal R(T)\\ \mathbf n_{r}=n}}\|v_{n_{\alpha_{1}}}\|_{2}^{2}\|v_{n_{\alpha_{2}}}\|_{2}^{2}\|v_{n_{\alpha_{3}}}\|_{2}^{2}\prod_{\beta\in B_{k}^{J}\setminus\{\alpha\}}\|v_{n_{\beta}}\|_{2}^{2}\ \prod_{\tilde{m}=1}^{r_{k}^{J}}|w_{n_{r_{\tilde{m}}}}|^{2}\Big)^{\frac1{2}}.$$
An easy calculation shows that the first sum behaves like $N^{-\frac{J-1}{2}+\frac{(J-2)}{200}+}$ and then by taking the $l^{2}(\Z)$ norm and use Young's inequality we arrive at 
$$N^{-\frac{J-1}{2}+\frac{(J-2)}{200}+}\ \|v\|_{L^{2}(\R)}^{b_{k}^{J}+2}\|w\|_{L^{2}(\T)}^{r_{k}^{J}}.$$
Similar considerations apply in the case that $\alpha\in R_{k}^{J}\subset T^{\infty}$ and give the upper bound 
$$N^{-\frac{J-1}{2}+\frac{(J-2)}{200}+}\ \|v\|_{L^{2}(\R)}^{b_{k}^{J}}\|w\|_{L^{2}(\T)}^{r_{k}^{J}+2}.$$

If $|\tilde\mu_{J+1}|\lesssim(2J+3)^{3}|\mu_{1}|^{1-\frac1{100}}$ holds in \eqref{sesee}, then for fixed $\mu_{j}$, $j=1,\ldots,J$, there are at most $O(|\mu_{1}|^{1-\frac1{100}})$ many choices for $\mu_{J+1}$. The same argument as above leads us to exactly the same expressions as in \eqref{hahah2} but with the first sum replaced by the following
$$\Big(\sum_{\substack{|\mu_{1}|>N\\ |\tilde\mu_{j}|>(2j+1)^{3}N^{1-\frac1{100}}\\ j=2,\ldots,J}}|\mu_{1}|^{1-\frac1{100}}\prod_{k=1}^{J}\frac1{|\tilde\mu_{k}|^{2}}|\mu_{k}|^{+}\Big)^{\frac1{2}},$$
which again is bounded from above by $N^{-\frac{J-1}{2}+\frac{(J-2)}{200}+}$ and the proof is complete.
\end{proof}

\begin{remark}
\label{reme}
For $s>0$ we have to observe that all previous lemmata hold true 
if we replace the $l^{2}L^{2}$ norm by the $l^{2}_{s}L^{2}$ 
norm and the $L^{2}(\R)$ norm by the $H^{s}(\R)$ norm. 
To see this, consider $n^{(j)}$ large. Then there exists at 
least one of $n_{1}^{(j)},n_{2}^{(j)},n_{3}^{(j)}$ such that 
$|n_{k}^{(j)}|\geq\frac13|n^{(j)}|$, $k\in\{1,2,3\}$, since 
we have the relation 
$n^{(j)}\approx n_{1}^{(j)}-n_{2}^{(j)}+n_{3}^{(j)}$. 
Therefore, in the estimates of the $J$th generation, there 
exists at least one frequency $n_{k}^{(j)}$ for some $j\in\{1,\ldots,J\}$ with the property
$$\langle n\rangle^{s}\leq 3^{js}\langle n_{k}^{(j)}\rangle^{s}\leq 3^{Js}\langle n_{k}^{(j)}\rangle ^{s}.$$
This exponential growth does not affect our calculations due to the double factorial decay in the denominator of \eqref{factori}.
\end{remark}

Before we finish this section let us state a lemma about the behaviour of the remainder operator $N_{2}^{(J)}$ as $J\to\infty$.

\begin{lemma}
\label{dankda}
Suppose that $w$ is a smooth periodic solution of \eqref{perpe} in $L^{2}(\T)$ such that its Fourier coefficients $\{w_{m}\}_{m\in\Z}\in l^{1}(\Z)$ and $v$ is a smooth solution of \eqref{nonperpe} such that $v\in M_{2,1}(\R)\subset L^{2}(\R)$. Then 
$$\lim_{J\to\infty}\|N_{2}^{(J+1)}(v)\|_{l^{2}(\Z)L^{2}(\R)}=0.$$
\end{lemma}
\begin{proof}
Obviously,
$$\|N_{2}^{(J+1)}(v)\|_{2}\leq \sum_{T\in C(J)}\sum_{\alpha\in T^{\infty}}\sum_{\substack{\mathbf n\in\mathcal R(T)\\ \mathbf n_{r}=n}}\|\tilde{q}^{J,t}_{T,\mathbf n}(\mathbf N_{1}^{t,\alpha}(\{v_{n_{b_{m}}}\}_{m=1}^{b_{k}^{J}},\{w_{n_{r_{\tilde{m}}}}\}_{\tilde{m}=1}^{r_{k}^{J}}))\|_{2}.$$
For a fixed tree $T=T_{k}^{J}\in C(J)$ assume that $\alpha\in B_{k}^{J}$. Using Lemma \ref{indu} we have the upper bound 
$$\sum_{\substack{\mathbf n\in\mathcal R(T)\\ \mathbf n_{r}=n}}\prod_{\beta\in B_{k}^{J}\setminus\{\alpha\}}\|v_{n_{\beta}}\|_{2}\ \frac{\|N_{1}^{t}(v)(n_{\alpha})\|_{2}}{\prod_{k=1}^{J}|\tilde{\mu}_{k}|}\ \prod_{\tilde{m}=1}^{r_{k}^{J}}|w_{n_{r_{\tilde{m}}}}|.$$
By the definition of the operator $N_{1}^{t}(v)$, see \eqref{main10},  and Remark \ref{expl}, we bound this further 
$$\sum_{\substack{\mathbf n\in\mathcal R(T)\\ \mathbf n_{r}=n}}\prod_{\beta\in B_{k}^{J}\setminus\{\alpha\}}\|v_{n_{\beta}}\|_{2}\ \prod_{\tilde{m}=1}^{r_{k}^{J}}|w_{n_{r_{\tilde{m}}}}|\Big(\sum_{\substack{n_{\alpha}\approx n_{\alpha_{1}}-n_{\alpha_{2}}+n_{\alpha_{3}} \\ n_{\alpha_{1}}\not\approx n_{\alpha}\not\approx n_{\alpha_{3}}}}\|v_{n_{\alpha_{1}}}\|_{2}\|v_{n_{\alpha_{2}}}\|_{2}\|v_{n_{\alpha_{3}}}\|_{2}+$$
$$\|v_{n_{\alpha_{1}}}\|_{2}|w_{n_{\alpha_{2}}}|\|v_{n_{\alpha_{3}}}\|_{2}+\|v_{n_{\alpha_{1}}}\|_{2}\|v_{n_{\alpha_{2}}}\|_{2}|w_{n_{\alpha_{3}}}|+|w_{n_{\alpha_{1}}}|\|v_{n_{\alpha_{2}}}\|_{2}|w_{n_{\alpha_{3}}}|+$$
$$|w_{n_{\alpha_{1}}}\|_{2}|w_{n_{\alpha_{2}}}|\|v_{n_{\alpha_{3}}}\|_{2}\Big)\ \frac1{\prod_{k=1}^{J}|\tilde{\mu}_{k}|}.$$
Let us treat only the sum that contains the quantity 
$\|v_{n_{\alpha_{1}}}\|_{2}|w_{n_{\alpha_{2}}}|\|v_{n_{\alpha_{3}}}\|_{2}$, 
the remaining terms can be treated in a similar manner. As in the proof of Lemma \ref{finaal}, H\"older's inequality implies the upper bound 
$$\frac1{(2J-1)!!^{\frac32}}\Big(\sum_{\substack{\mathbf n\in\mathcal R(T)\\ \mathbf n_{r}=n}}\prod_{\beta\in B_{k}^{J}\setminus\{\alpha\}}\|v_{n_{\beta}}\|_{2}^{2}\prod_{\tilde{m}=1}^{r_{k}^{J}}|w_{n_{r_{\tilde{m}}}}|^{2}\Big(\sum_{\substack{n_{\alpha}\approx n_{\alpha_{1}}-n_{\alpha_{2}}+n_{\alpha_{3}} \\ n_{\alpha_{1}}\not\approx n_{\alpha}\not\approx n_{\alpha_{3}}}}\|v_{n_{\alpha_{1}}}\|_{2}|w_{n_{\alpha_{2}}}|\|v_{n_{\alpha_{3}}}\|_{2}\Big)^{2}\Big)^{\frac1{2}}.$$
Then by taking the $l^{2}(\Z)$ norm we arrive at 
$$\frac1{(2J-1)!!^{\frac32}}\Big(\sum_{n\in\Z}\ \sum_{\substack{\mathbf n\in\mathcal R(T)\\ \mathbf n_{r}=n}}\prod_{\beta\in B_{k}^{J}\setminus\{\alpha\}}\|v_{n_{\beta}}\|_{2}^{2}\prod_{\tilde{m}=1}^{r_{k}^{J}}|w_{n_{r_{\tilde{m}}}}|^{2}(\{\|v_{n_{\alpha_{1}}}\|_{2}\}\ast\{|w_{n_{\alpha_{2}}}|\}\ast\{\|v_{n_{\alpha_{3}}}\|_{2}\})^{2}(n_{\alpha})\Big)^{\frac1{2}} ,$$
applying Young's inequality in $l^{1}(\Z)$ for $2J+1$ sequences we get
$$\frac1{(2J-1)!!^{\frac32}}\ \|v\|_{L^{2}(\R)}^{b_{k}^{J}-1}\ \|w\|_{L^{2}(\T)}^{r_{k}^{J}}\ \|\{\|v_{n_{\alpha_{1}}}\|_{2}\}\ast\{|w_{n_{\alpha_{2}}}|\}\ast\{\|v_{n_{\alpha_{3}}}\|_{2}\}\|_{l^{2}},$$
and, again using Young's inequality together with the embedding $M_{2,1}(\R)\hookrightarrow L^{2}(\R)$ and the assumption that the Fourier coefficients of $w$ are in $l^{1}(\Z)$, this  implies the upper bound  
$$\frac1{(2J-1)!!^{\frac32}}\ \|v\|_{M_{2,1}}^{b_{k}^{J}-1}\ \|\{w_{m}\}_{m\in\Z}\}\|^{r_{k}^{J}}_{l^{1}(\Z)}\ \|v\|_{M_{2,1}}^{2}\ \|\{w_{m}\}_{m\in\Z}\|_{l^{1}(\Z)}=\frac{\|v\|_{M_{2,1}}^{b_{k}^{J}+1}\|\{w_{m}\}_{m\in\Z}\|_{l^{1}}^{r_{k}^{J}+1}}{(2J-1)!!^{\frac32}}.$$
Similar estimates apply in the case  $\alpha\in R_{k}^{J}$.

Finally, by adding up all these expressions for every different colored tree $T\in C(J)$, see \eqref{growthoftrees}, we get 
$$\|N_{2}^{(J+1)}(v)\|_{l^{2}(\Z)L^{2}(\R)}\lesssim\frac{5^{J}}{(2J-1)!!^{\frac12}}\ (\|v\|_{M_{2,1}}+\|\{w_{m}\}_{m\in\Z}\|_{l^{1}(\Z)})^{2J+3},$$
which goes to zero as $J\to\infty$. So the proof is complete.  
\end{proof}

\end{section}

\begin{section}{existence of weak solutions in the extended sense}
\label{weakextended}
In this subsection we prove Theorem \ref{th1}. The calculations are the similar as in \cite{GKO}, \cite{CHKP}, \cite{NP}, however, with the additional difficulty that we have to handle mixed continuous and discrete variables. 
For this reason we only mention the basic steps of the argument, concentrating mainly on the important differences. 

We start by defining the partial sum operator $\Gamma_{v_{0}}^{(J)}$ as
\begin{equation}
\label{gamaa}
\Gamma_{v_{0}}^{(J)}v(t)=v_{0}+\sum_{j=2}^{J}N_{0}^{(j)}(v)(n)-\sum_{j=2}^{J}N_{0}^{(j)}(v_{0})(n)
\end{equation}
$$+\int_{0}^{t}R_{1}^{\tau}(v)(n)+R_{2}^{\tau}(v)(n)+\sum_{j=2}^{J}N_{r}^{(j)}(v)(n)+\sum_{j=1}^{J}N_{1}^{(j)}(v)(n)\ d\tau,$$
where we have $N_{1}^{(1)}\coloneqq N_{11}^{t}$ from \eqref{main13}, $N_{0}^{(2)}\coloneqq N_{21}^{t}$ from \eqref{nex} and $v_{0}\in H^{s_{1}}(\R)$ is our initial data. Here we assume that we have smooth solutions (see Section \ref{smoothsection} so that all calculations of Sections \ref{firststeps} and \ref{treenotation} are applicable. Moreover, let us state that all operators appearing in the definition of $\Gamma_{v_{0}}^{(J)}v(t)$ depend also on the fixed function $w\in X_{T_{0}}(\T)=C([0,T_{0}],H^{s_{2}}(\T))$ that is the solution of \eqref{perpe} with initial data $w_{0}\in H^{s_{2}}(\T)$. For this $w$ we know that 
\begin{equation}
\label{maybeineed}
\|w\|_{X_{T_{0}}(\T)}\lesssim\|w_{0}\|_{H^{s_{2}}(\T)}.
\end{equation}

In the following we will denote by $X_{T}(\R)=C([0,T],H^{s_{1}}(\R))$. Our goal is to show that the series appearing on the RHS of \eqref{gamaa} converge absolutely in $X_{T}(\R)$ for sufficiently small $T>0,$ if $v\in X_{T}(\R),$ even for $J=\infty$. Indeed, by Lemmata \ref{lem}, \ref{lemle}, \ref{finaal}, and \ref{finaal2} we obtain 
\begin{equation}
\label{argg}
\|\Gamma_{v_{0}}^{(J)}v\|_{X_{T}(\R)}\leq\|v_{0}\|_{H^{s_{1}}(\R)}+C\sum_{j=2}^{J}N^{-\frac{j-1}{2}+\frac{j-2}{200}+}(\|v\|_{X_{T}(\R)}^{2j-1}+\|v_{0}\|_{H^{s_{1}}(\R)}^{2j-1}+\|w\|_{X_{T_{0}}(\T)}^{2j-1})
\end{equation}
$$+CT\Big[\|v\|^{3}_{X_{T}(\R)}+\|w\|_{X_{T_{0}}(\T)}^{3}+\sum_{j=2}^{J}N^{-\frac{j-1}{2}+\frac{j-2}{200}+}(\|v\|_{X_{T}(\R)}^{2j+1}+\|w\|_{X_{T_{0}}(\T)}^{2j+1})+$$
$$N^{\frac1{2}+}(\|v\|_{X_{T}(\R)}^{3}+\|w\|_{X_{T_{0}}(\T)}^{3})+\sum_{j=2}^{J}N^{-\frac{j-2}{2}+\frac{j-3}{200}+}(\|v\|_{X_{T}(\R)}^{2j+1}+\|w\|_{X_{T_{0}}(\T)}^{2j+1})\Big].$$
From \eqref{maybeineed} we estimate $\|w\|_{X_{T_{0}}(\T)}$ by $\|w_{0}\|_{H^{s_{2}}(\T)}$ and assuming that the sum $\|v_{0}\|_{H^{s_{1}}(\R)}+\|w_{0}\|_{H^{s_{2}}(\T)}\leq R$ and $\|v\|_{X_{T}(\R)}\leq\tilde R$, with $\tilde R\geq R\geq1$ we may continue from \eqref{argg} in exactly the same way as in \cite{GKO}, \cite{CHKP}, \cite{NP} to show that for sufficiently large $N$ and sufficiently small $T=T(\|v_{0}\|_{H^{s_{1}}(\R)}+\|w_{0}\|_{H^{s_{2}}(\T)})>0$ the partial sum operators $\Gamma_{v_{0}}^{(J)}$ are well defined in $X_{T}(\R)$, for every $J\in\mathbb N\cup\{\infty\}$. We will write $\Gamma_{v_{0}}$ for $\Gamma_{v_{0}}^{(\infty)}$.

Our next step is, given an initial data 
$u_{0}=v_{0}+w_{0}\in H^{s_{1}}(\R)+H^{s_{2}}(\T)$, to construct 
a solution $u$ with the properties claimed in Theorem \ref{th1}. 
We start with the periodic part $w_{0}$. As it was done in 
\cite{GKO} we approximate $w_{0}$ by smooth initial data 
$w_{0}^{(m)}\in H^{\infty}(\T)$ with
\begin{equation}
\label{wapprox}
\lim_{m\to\infty}w_{0}^{(m)}=w_{0},\ \mbox{in}\ H^{s_{2}}(\T).
\end{equation}
For such initial data $w_{0}^{(m)}$ we know that we can find smooth solution $w^{(m)}$ of NLS \eqref{perpe} in $C([0,T],H^{s_{2}}(\T))$ that satisfies Duhamel's formulation
\begin{equation}
\label{wduhamel}
w^{(m)}=w_{0}^{(m)}\pm\int_{0}^{t}S(-\tau)[|S(\tau)w^{(m)}|^{2}S(\tau)w^{(m)}]\ d\tau
\end{equation}
and from \cite{GKO} it follows that there is a common time of existence $T_{0}=T_{0}(\|w_{0}\|_{H^{s_{2}}(\T)})$ for all solutions $w^{(m)}$. In addition, they show that the sequence $\{w^{(m)}\}_{m\in\mathbb N}$ is Cauchy in $X_{T_{0}}(\T)=C([0,T_{0}],H^{s_{2}}(\T))$ and that the limit function $w\in X_{T_{0}}(\T)$ satisfies NLS \eqref{perpe} in the sense of Definition \ref{def6}.   

We also approximate $v_{0}$ by smooth functions $v_{0}^{(m)}\in H^{s_{1}}(\R),$ so that
\begin{equation}
\label{vapprox}
\lim_{m\to\infty}v_{0}^{(m)}=v_{0},\ \mbox{in}\ H^{s_{1}}(\R),
\end{equation}
and by Section \ref{smoothsection} we may find smooth solutions $v^{(m)}$ of \eqref{nonperpe} in $X_{T}(\R)=C([0,T], H^{s_{1}}(\R))$ that satisfy Duhamel's formulation
\begin{equation}
\label{argg7}
v^{(m)}=v_{0}^{(m)}\pm\int_{0}^{t}S(-\tau)[G(S(\tau)w^{(m)},S(\tau)v^{(m)})]\ d\tau=
\end{equation}
$$v_{0}^{(m)}+\sum_{j=2}^{\infty}N_{0}^{(j)}(v^{(m)})(n)-\sum_{j=2}^{\infty}N_{0}^{(j)}(v_{0}^{(m)})(n)$$
$$+\int_{0}^{t}R_{1}^{\tau}(v^{(m)})(n)+R_{2}^{\tau}(v^{(m)})(n)+\sum_{j=2}^{\infty}N_{r}^{(j)}(v^{(m)})(n)+\sum_{j=1}^{\infty}N_{1}^{(j)}(v^{(m)})(n)\ d\tau=\Gamma_{v_{0}^{(m)}}v^{(m)},$$
where we used Lemma \ref{dankda}, namely that the remainder operator goes to zero as $J\to\infty$. From this, following exactly the same arguments as in \cite{GKO}, \cite{CHKP}, \cite{NP} we can prove that \eqref{argg7} holds in $X_{T_{0}}(\R)$ for the same time $T_{0}=T_{0}(R)>0$ independent of $m\in\mathbb N$ and also that 
\begin{equation}
\label{argg10}
\|v^{(m_{1})}-v^{(m_{2})}\|_{X_{T_{0}(\R)}}=\|\Gamma_{v_{0}^{(m_{1})}}v^{(m_{1})}-\Gamma_{v_{0}^{(m_{2})}}v^{(m_{2})}\|_{X_{T_{0}}(\R)}\leq c\ \|v_{0}^{(m_{1})}-v_{0}^{(m_{2})}\|_{H^{s_{1}}(\R)}
\end{equation}
for some constant $c>0$. 
Therefore, the sequence $\{v^{(m)}\}_{m\in\mathbb N}$ is Cauchy in the Banach space $X_{T_{0}}(\R)$, we denote by $v^{\infty}$ its limit in $X_{T_{0}}(\R)$. 

We will show that 
$V^{\infty}=S(t)v^{\infty}$ satisfies NLS \eqref{nonperpe} in the sense of Definition \ref{def3}. For convenience, we drop the superscript $\infty$ and write $V$, and $v$. In addition, let $V^{(m)}\coloneqq S(t)v^{(m)}, W^{(m)}=S(t)w^{(m)}$ and 
$W=S(t)w$. Obviously, $V^{(m)}\to V$ in $X_{T_{0}}(\R)$, because $v^{(m)}\to v$ in $X_{T_{0}}(\R)$, and similarly 
$W^{(m)}\to W$ in $X_{T_{0}}(\T)$ since $w^{(m)}\to w$ in 
$X_{T_{0}}(\T)$. 
Thus, $\partial_{x}V^{(m)}\to\partial_{x} V, \partial_{t}V^{(m)}\to\partial_{t}V$ and $\partial_{x}W^{(m)}\to\partial_{x}W$, 
 $\partial_{t}W^{(m)}\to\partial_{t}W$ in the sense of distributions. Since $V^{(m)}$ satisfies \eqref{nonperpe} and $W^{(m)}$ satisfies \eqref{perpe} for every $m\in\mathbb N$, 
 we have that 
\begin{equation}
\label{perconvergdist}
\mathcal N(W^{(m)})=|W^{(m)}|^{2}W^{(m)}=-i\partial_{t}W^{(m)}+\partial_{x}^{2}W^{(m)}
\end{equation}
converges to some distribution $\tilde{w}$, which is equal to 
$\mathcal N(W)$ interpreted in the sense of 
Definition \ref{def6}, as it was shown in \cite{GKO}.
and
\begin{equation}
\label{convergdist}
G(W^{(m)},V^{(m)})=|W^{(m)}+V^{(m)}|^{2}(W^{(m)}+V^{(m)})-|W^{(m)}|^{2}W^{(m)}=-i\partial_{t}V^{(m)}+\partial_{x}^{2}V^{(m)}
\end{equation}
converges to some distribution $\tilde{v}$. Our claim is the following.

\begin{proposition}
\label{argg12}
Let $\tilde{v}$ be the limit of $G(W^{(m)},V^{(m)})$ in the sense of distributions as $m\to\infty$. Then $\tilde{v}=G(W,V)$ where $G(W,V)$ is to be interpreted in the sense of Definition \ref{def3}.
\end{proposition}
\begin{proof}
Consider a sequence of Fourier cutoff multipliers $\{T_{N}\}_{N\in\mathbb N}$ as in Definition \ref{def1}. We will prove that 
$$\lim_{N\to\infty} G(T_{N}W,T_{N}V)=\tilde{v}$$
in the sense of distributions. Let $\phi$ be a test function and $\epsilon>0$ a fixed given number. Our goal is to find $N_{0}\in\mathbb N$ such that for all $N\geq N_{0}$ we have
\begin{equation}
\label{argg13}
|\langle \tilde{v}- G(T_{N}W,T_{N}V), \phi\rangle|<3\epsilon.
\end{equation}
The LHS can be estimated by 
\begin{equation}
\label{nownowg}
|\langle \tilde{v}-G(W^{(m)},V^{(m)}),\phi\rangle|+|\langle G(W^{(m)},V^{(m)})-G(T_{N}W^{(m)},T_{N}V^{(m)}),\phi\rangle|+
\end{equation}
$$|\langle G(T_{N}W^{(m)},T_{N}V^{(m)})-G(T_{N}W,T_{N}V),\phi\rangle|.$$
The first term is estimated very easily since by the definition of $\tilde{v}$ we have that
\begin{equation}
\label{argg14}
|\langle \tilde{v}- G(W^{(m)},V^{(m)}), \phi\rangle|< \epsilon,
\end{equation}
for sufficiently large $m\in\mathbb N$.

To continue, we consider the second summand of \eqref{nownowg} 
for fixed $m$. Writing out the difference, we see that we have 
to estimate five expressions
\begin{multline*}
\langle |V^{(m)}|^{2}V^{(m)}-|T_{N}V^{(m)}|^{2}T_{N}V^{(m)},\phi\rangle+\langle (W^{(m)})^{2}\overline{V^{(m)}}-(T_{N}W^{(m)})^{2}\overline{T_{N}V^{(m)}},\phi\rangle \\
+\langle (V^{(m)})^{2}\overline{W^{(m)}}-(T_{N}V^{(m)})^{2}\overline{T_{N}W^{(m)}},\phi\rangle+2\ \langle |W^{(m)}|^{2}V^{(m)}-|T_{N}W^{(m)}|^{2}T_{N}V^{(m)},\phi\rangle\\
+2\ \langle |V^{(m)}|^{2}W^{(m)}-|T_{N}V^{(m)}|^{2}T_{N}W^{(m)},\phi\rangle.
\end{multline*}
The first was estimated in \cite{CHKP} and \cite{NP}. For the second term we note 
\begin{multline*} 
\Big|\int\int (W^{(m)})^{2}\overline{(Id-T_{N})V^{(m)}}\phi+\overline{T_{N}V^{(m)}}(W^{(m)}-T_{N}W^{(m)})(W^{(m)}+T_{N}W^{(m)})\phi\Big| \\
\le \|W^{(m)}\|_{L^{\infty}_{T,x}}^{2}\|(Id-T_{N})V^{(m)}\|_{L^{2}_{T,x}}\|\phi\|_{L^{2}_{T,x}} 
	+\|T_{N}V^{(m)}\|_{L^{\infty}_{T,x}}\|W^{(m)}\\
+T_{N}W^{(m)}\|_{L^{\infty}_{T,x}}\int\int|W^{(m)}-T_{N}W^{(m)}||\phi|.
\end{multline*}
The integral term can be written as 
\begin{align*}
\int_{0}^{T}\sum_{k\in\Z}\ \int_{k}^{k+1}|W^{(m)}-T_{N}W^{(m)}||\phi|
	&\leq\int_{0}^{T}\sum_{k\in\Z}\|(Id-T_{N})W^{(m)}\|_{L^{2}(k,k+1)}\|\phi\|_{L^{2}(k,k+1)}\\ 
&= \int_{0}^{T}\|(Id-T_{N})W^{(m)}\|_{L^{2}(\T)}\sum_{k\in\Z}\|\phi\|_{L^{2}(k,k+1)}\, ,
\end{align*}
which is bounded from above by
$$\|(Id-T_{N})W^{(m)}\|_{L^{2}_{T,x}}\|t\to\sum_{k\in\Z}\|\phi(t,\cdot)\|_{L^{2}(k,k+1)}\|_{L^{2}(0,T)}.$$
Therefore, for the second term we have the estimate
\begin{equation}
\label{perandcont}
C_{\phi,m}\Big(\|(Id-T_{N})V^{(m)}\|_{L^{2}([0,T],L^{2}(\R))}+\|(Id-T_{N})W^{(m)}\|_{L^{2}([0,T],L^{2}(\T))}\Big)
\end{equation}
which tends to zero as $N\to\infty$ by the definition of the Fourier cutoff operators and the Dominated Convergence Theorem. For the third term we have to consider the quantities
$$\Big|\int\int(V^{(m)})^{2}\overline{W^{(m)}-T_{N}W^{(m)}}\phi+\overline{T_{N}W^{(m)}}(V^{(m)}-T_{N}V^{(m)})(V^{(m)}+T_{N}V^{(m)})\phi\Big|.$$
Doing the same as for the previous term, we obtain an expression 
analog  to  \eqref{perandcont}. We treat the forth and fifth terms similarly. This allows us to choose $N_{0}=N_{0}(m)>0$ with the property
\begin{equation}
\label{argg15}
C_{\phi, m}\Big(\|(Id-T_{N})V^{(m)}\|_{L^{2}([0,T],L^{2}(\R))}+\|(Id-T_{N})W^{(m)}\|_{L^{2}([0,T],L^{2}(\T))}\Big)< \epsilon,
\end{equation}
for all $N\geq N_{0}$. 

For the last term of \eqref{nownowg} we need to observe two things. Firstly, by applying the iteration process (see also \cite{GKO}, \cite{CHKP} and \cite{NP}) that we described in Sections \ref{firststeps} and \ref{treenotation} we see that $\{G(W^{(m)},V^{(m)})\}_{m\in\mathbb N}$ is Cauchy in $S'((0,T)\times\R)$ as $m\to\infty$ for each fixed $N,$ since the sequences $V^{(m)},W^{(m)}$ are Cauchy in $C([0,T],H^{s_{1}}(\R))$ and $C([0,T],H^{s_{2}}(\T))$ respectively. Because the multipliers $m_{N}$ of $T_{N}$ are uniformly bounded we conclude that this convergence is uniform in $N$. 

Secondly, for fixed $N,$ $T_{N}V\in C([0,T],H^{\infty}(\R))$ and $T_{N}W\in C([0,T],H^{\infty}(\T))$ since $V\in H^{s_{1}}(\R), W\in H^{s_{2}}(\T)$ and the multiplier $m_{N}$ of $T_{N}$ is compactly supported. Hence 
$$G(T_{N}W,T_{N}V)=|T_{N}V|^{2}T_{N}V+(T_{N}W)^{2}\overline{T_{N}V}+(T_{N}V)^{2}\overline{T_{N}W}+2|T_{N}W|^{2}T_{N}V+2|T_{N}V|^{2}T_{N}W$$
makes sense as a function. Then we have to estimate the following five summands
$$\langle |T_{N}V^{(m)}|^{2}T_{N}V^{(m)}-|T_{N}V|^{2}T_{N}V,\phi\rangle+\langle (T_{N}W^{(m)})^{2}\overline{T_{N}V^{(m)}}-(T_{N}W)^{2}\overline{T_{N}V},\phi\rangle+$$
$$\langle (T_{N}V^{(m)})^{2}\overline{T_{N}W^{(m)}}-(T_{N}V)^{2}\overline{T_{N}W},\phi\rangle+2\langle |T_{N}W^{(m)}|^{2}T_{N}V^{(m)}-|T_{N}W|^{2}T_{N}V,\phi\rangle+$$
$$2\langle |T_{N}V^{(m)}|^{2}T_{N}W^{(m)}-|T_{N}V|^{2}T_{N}W,\phi\rangle.$$
The first term was estimated in \cite{CHKP} and \cite{NP}. For the second term we have to bound  
$$\Big|\int\int(T_{N}W^{(m)})^{2}\overline{(T_{N}(V^{(m)}-V))}\phi+\overline{T_{N}V}(T_{N}W^{(m)}-T_{N}W)(T_{N}W^{(m)}+T_{N}W)\phi\Big|\leq$$
$$\|T_{N}W^{(m)}\|^{2}_{L^{\infty}_{T,x}}\|T_{N}(V^{(m)}-V)\|_{L^{2}_{T,x}}\|\phi\|_{L^{2}_{T,x}}+$$
$$\int_{0}^{T}\sum_{k\in\Z}\ \int_{k}^{k+1}\Big|T_{N}V(T_{N}W^{(m)}-T_{N}W)(T_{N}W^{(m)}+T_{N}W)\phi\Big|.$$
The second expression is bounded from above by
$$\int_{0}^{T}\sum_{k\in\Z}\|T_{N}V\|_{L^{4}(k,k+1)}\|T_{N}(W^{(m)}-W)\|_{L^{4}(k,k+1)}\|T_{N}W^{(m)}+T_{N}W\|_{L^{4}(k,k+1)}\|\phi\|_{L^{4}(k,k+1)}$$
which is less than
$$\int_{0}^{T}\|T_{N}V\|_{L^{4}(\R)}\|T_{N}(W^{(m)}-W)\|_{L^{4}(\T)}\|T_{N}W^{(m)}+T_{N}W\|_{L^{4}(\T)}\sum_{k\in\Z}\|\phi\|_{L^{4}(k,k+1)}\leq$$
$$\|T_{N}V\|_{L^{4}_{T,x}}\|T_{N}W^{(m)}+T_{N}W\|_{L^{4}_{T,x}}\|T_{N}(W^{(m)}-W)\|_{L^{4}_{T,x}}\|t\to\sum_{k\in\Z}\|\phi\|_{L^{4}(k,k+1)}\|_{L^{4}(0,T)}.$$
Then we use H\"older's inequality in the interval $(0,T)$ to pass 
from the $L^{4}$ norm to the $L^{\infty}$ norm and in the space 
variable an application of Parseval's identity, together with the 
fact that the multiplier operators $T_{N}$ have compactly 
supported symbols $m_{N}$, implies the bound
$$C_{\phi,\|V\|_{X_{T}(\R)},\|W\|_{X_{T}(\T)}}M^{\frac34}T^{\frac34}\|W^{(m)}-W\|_{X_{T}(\T)}<\epsilon, $$
where the number $M=M(N)>0$ is chosen so that $\mbox{supp}m_{N}\subset[-M,M]$. For the third term we have to estimate the quantity
$$\Big|\int\int(T_{N}V^{(m)})^{2}\overline{(T_{N}(W^{(m)}-W))}\phi+\overline{T_{N}W}(T_{N}(V^{(m)}-V))(T_{N}V^{(m)}+T_{N}V)\phi\Big|,$$
for which similar bounds apply as for the previous term. The same holds for the forth and fifth terms. 

From these observations we derive that $G(T_{N}W^{(m)},T_{N}V^{(m)})\to G(T_{N}W,T_{N}V)$ in the space $S'((0,T)\times\R)$ as $m\to\infty$ uniformly in $N$. Equivalently, 
\begin{equation}
\label{argg16}
|\langle G(T_{N}W^{(m)},T_{N}V^{(m)})- G(T_{N}W,T_{N}V), \phi\rangle|<\epsilon,
\end{equation}
for all large $m$, uniformly in $N$. Therefore, \eqref{argg13} follows by choosing $m$ sufficiently large so that \eqref{argg14} and \eqref{argg16} hold, and then choosing $N_{0}=N_{0}(m)$ such that \eqref{argg15} holds. 

\end{proof}
Finally, we have shown that the function $V=V^{\infty}$ is a solution of NLS \eqref{nonperpe} in the sense of Definition \ref{def3}. 

\end{section}

\begin{section}{unconditional uniqueness of solutions}
\label{uniquenessofsol}
In this section we prove Theorem \ref{finth2}. Let us assume that the initial condition 
$u_{0}=v_{0}+w_{0}\in H^{s}(\R)+H^{\frac12+\epsilon}(\T)$ where 
$\frac16\leq s\leq\frac12$ and $\epsilon>0$. Notice that for such 
$s$ we have the embeddings $H^{s}(\R)\hookrightarrow L^{3}(\R)$ and 
$H^{\frac12+\epsilon}(\T)\hookrightarrow L^{\infty}(\T)$. 
Therefore, if $V$ is a solution of NLS \eqref{nonperpe} in 
$C([0,T], H^{s}(\R))$, then $V$ and hence 
$v=e^{it\partial_{x}^{2}}V$ are elements of 
$C([0,T], H^{s}(\R))\hookrightarrow C([0,T], L^{3}(\R))$. 
Similarly, for $W$ being a solution of NLS \eqref{perpe} in 
$C([0,T], H^{\frac12+\epsilon}(\T))$, we have  
$w=e^{it\partial_{x}^{2}}W\in C([0,T], H^{\frac12+\epsilon}(\T))\hookrightarrow C([0,T], L^{\infty}(\T))$. 

Therefore, the nonlinearity $G(w,v)$ makes sense as a function in $L^{1}(\R)+L^{2}(\R)$ since $|v|^{2}v\in L^{1}(\R), w^{2}\bar{v}, |w|^{2}v\in L^{3}(\R)\cap L^{2}(\R)$ and $v^{2}\bar{w}, |v|^{2}w\in L^{1}(\R)\cap L^{\frac32}(\R)$. 

As a consequence of this,  its box operator $\Box_{n}G(w,v)\in L^{2}(\R)$ and from the PDE  
\begin{equation}
\label{godd}
i\partial_{t}v_{n}=S(t)\Box_{n}G(S(-t)w,S(-t)v)\, ,
\end{equation}
which is true in the sense of distributions $(C^{\infty}([0,T],S(\R)))'$, we infer $\partial_{t}v_{n}\in C([0,T],L^{2}(\R))$. 

This, together with $v_{n}\in C([0,T], L^{2}(\R))$,  
already implies $v_{n}\in C^{1}([0,T],L^{2}(\R))$. Indeed, to obtain this it suffices to know that if two space--time distributions 
$S$ and $T\in (C^{\infty}([0,T],S(\R)))'$ have the same time derivatives, $\partial_{t}S=\partial_{t}T$, then there is distribution $c$, acting only on the space variable, such that $S=T+c$. This can be found, for example,  in \cite[Section 3.3]{Vlad}. 

Thus, we can rewrite the the PDE in the integral form 
\begin{equation}
\label{allell}
v_{n}=v_{n}(0)+i\int_{0}^{t}S(\tau)\Box_{n}G(S(-\tau)w,S(-\tau)v)\ d\tau\, ,
\end{equation}
which means that we can continue with the differentiation by parts technique, as it was described in Sections \ref{firststeps} and \ref{treenotation}, directly for the function $v$ without having to approximate it by smooth solutions, as done in the previous Section \ref{weakextended}. The next lemma justifies the interchange of time differentiation and space integration

\begin{lemma}
\label{didi}
Let $f(t,x),\partial_{t}f(t,x)\in C([0,T],L^{1}(\R^{d}))$ and define the distribution $\int_{\R^{d}}f(\cdot,x)dx$ by
$$\Big\langle \int_{\R^{d}}f(\cdot, x)dx, \phi\Big\rangle=\int_{\R}\int_{\R^{d}}f(t,x)\phi(t)dxdt,$$
with $\phi\in C^{\infty}_{c}(\R)$. Then, $\partial_{t}\int_{\R^{d}}f(\cdot,x)dx=\int_{\R^{d}}\partial_{t}f(\cdot,x)dx$.
\end{lemma}
\begin{proof}
By definition
$$\Big\langle\partial_{t}\int_{\R^{d}}f(\cdot,x)dx,\phi\Big\rangle=-\Big\langle\int_{\R^{d}}f(\cdot,x)dx,\phi'\Big\rangle=-\int_{\R}\int_{\R^{d}}f(t,x)\phi'(t)dxdt$$
and, since $f\in C([0,T],L^{1}(\R^{d}))$, we can change the order of integration by Fubini's Theorem to obtain
$$-\int_{\R^{d}}\int_{\R}f(t,x)\phi'(t)dtdx=\int_{\R^{d}}\int_{\R}\partial_{t}f(t,x)\phi(t)dtdx=\int_{\R}\int_{\R^{d}}\partial_{t}f(t,x)\phi(t)dxdt,$$
where in the first equality we used the definition of the weak derivative of $f$ and in the second equality Fubini's Theorem with the fact that $\partial_{t}f\in C([0,T],L^{1}(\R^{d}))$. The last integral is equal to 
$$\Big\langle\int_{\R^{d}}\partial_{t}f(\cdot,x)dx,\phi\Big\rangle$$
and the proof is complete.
\end{proof}
Consider now the expressions \eqref{ttr1}, \eqref{ttr4} and \eqref{ttr5} for fixed $n$ and $\xi$. We want to apply Lemma \ref{didi} to each one of the following functions
$$f_{1}(t,\xi_{1},\xi_{3})=\sigma_{n}(\xi)\frac{e^{-2it(\xi-\xi_{1})(\xi-\xi_{3})}}{-2i(\xi-\xi_{1})(\xi-\xi_{3})}\ \hat{v}_{n_{1}}(\xi_{1})\hat{\bar{v}}_{n_{2}}(\xi-\xi_{1}-\xi_{3})\hat{v}_{n_{3}}(\xi_{3}),$$
$$f_{2}(t, \xi_{1})=\sigma_{n}(\xi)w_{n_{3}}\frac{e^{-2it(\xi-n_{3})(\xi-\xi_{1})}}{-2i(\xi-n_{3})(\xi-\xi_{1})}\hat{v}_{n_{1}}(\xi_{1})\hat{\bar{v}}_{n_{2}}(\xi-\xi_{1}-n_{3}),$$
$$f_{3}(t,\xi_{1})=\sigma_{n}(\xi)\bar{w}_{n_{2}}\frac{e^{-2it(\xi-\xi_{1})(\xi_{1}-n_{2})}}{-2i(\xi-\xi_{1})(\xi_{1}-n_{2})}\hat{v}_{n_{1}}(\xi_{1})\hat{v}_{n_{3}}(\xi-\xi_{1}+n_{2}),$$
where $\xi\approx n, \xi_{1}\approx n_{1},\xi_{3}\approx n_{3}, \xi-\xi_{1}-\xi_{3}\approx -n_{2}$ and $(n,n_{1},n_{2},n_{3})\in A_{N}(n)^{c}$ 
given by \eqref{idid}. With the use of Young's inequality and the fact that for all 
$n$, $\hat{v}_{n}, \partial_{t}\hat{v}_{n}$ are compactly supported functions in $L^{2}(\R)$, 
it is not hard to obtain that $f_{1},\partial_{t}f_{1}\in C([0,T],L^{1}(\R^2))$ and $f_{2},f_{3},\partial_{t}f_{2},\partial_{t}f_{3}\in C([0,T],L^{1}(\R))$. Thus, for $f_{1}$, and  similarly for $f_{2},f_{3}$,  
\begin{multline*}
\partial_{t}\Big[\int_{\R^2}\sigma_{n}(\xi)\frac{e^{-2it(\xi-\xi_{1})(\xi-\xi_{3})}}{-2i(\xi-\xi_{1})(\xi-\xi_{3})}\ \hat{v}_{n_{1}}(\xi_{1})\hat{\bar{v}}_{n_{2}}(\xi-\xi_{1}-\xi_{3})\hat{v}_{n_{3}}(\xi_{3})d\xi_{1}d\xi_{3}\Big] \\
=\int_{\R^2}\sigma_{n}(\xi)\partial_{t}\Big[\sigma_{n}(\xi)\frac{e^{-2it(\xi-\xi_{1})(\xi-\xi_{3})}}{-2i(\xi-\xi_{1})(\xi-\xi_{3})}\ \hat{v}_{n_{1}}(\xi_{1})\hat{\bar{v}}_{n_{2}}(\xi-\xi_{1}-\xi_{3})\hat{v}_{n_{3}}(\xi_{3})\Big]d\xi_{1}d\xi_{3} \\
=\int_{\R^2}\sigma_{n}(\xi)\partial_{t}\Big[\frac{e^{-2it(\xi-\xi_{1})(\xi-\xi_{3})}}{-2i(\xi-\xi_{1})(\xi-\xi_{3})}\Big]\hat{v}_{n_{1}}(\xi_{1})\hat{\bar{v}}_{n_{2}}(\xi-\xi_{1}-\xi_{3})\hat{v}_{n_{3}}(\xi_{3})d\xi_{1}d\xi_{3} \\
+\int_{\R^2}\sigma_{n}(\xi)\frac{e^{-2it(\xi-\xi_{1})(\xi-\xi_{3})}}{-2i(\xi-\xi_{1})(\xi-\xi_{3})}\partial_{t}\Big[\hat{v}_{n_{1}}(\xi_{1})\hat{\bar{v}}_{n_{2}}(\xi-\xi_{1}-\xi_{3})\hat{v}_{n_{3}}(\xi_{3})\Big]d\xi_{1}d\xi_{3}.
\end{multline*}
In the second equality we used the product rule which is applicable since $\hat{v}_{n}\in C^{1}([0,T],L^{2}(\R))$.

Finally it remains to justify the interchange of differentiation in time and summation in the discrete variable but this is done in exactly the same way as in \cite{GKO} (Lemma $5.1$). Similar arguments justify the interchange on the $J$th step of the infinite iteration procedure. 

Thus, we obtain the following expression in $C([0,T],H^{s}(\R))$ for the solution $v$ of NLS \eqref{godd} with initial data $v_{0}$
\begin{equation} 
\label{antt1}
v=\Gamma_{v_{0}}v+\lim_{J\to\infty}\int_{0}^{t}N_{2}^{(J+1)}(v)d\tau,
\end{equation}
where the limit is an element of $C([0,T],H^{s}(\R))$. Its existence follows from the fact that the operators $\Gamma_{v_{0}}^{(J)}v$ converge to $\Gamma_{v_{0}}v$ in the norm of $C([0,T],H^{s}(\R))$ as $J\to\infty$. The important estimate about the remainder operator $N_{2}^{(J)}$ is the following

\begin{lemma}
\label{finafinafina}
$$\lim_{J\to\infty}\|N_{2}^{(J)}(v)\|_{l^{\infty}L^{2}(\R)}=0.$$
\end{lemma} 
The proof is very similar to the one given in \cite{NP}, Lemma 28, where we have to consider the cases $\partial_{t}v_{n}, \partial_{t}w_{n}$ with similar arguments. This lemma implies that $\lim_{J\to\infty}\int_{0}^{t}N_{2}^{(J+1)}(v)d\tau$ is equal to $0$ in $X(T)=C([0,T],H^{s}(\R))$. From this we obtain the uniqueness of NLS \eqref{godd} since if there are two solutions $v_{1}$ and $v_{2}$ with the same initial datum $v_{0}$ we obtain by \eqref{argg10}
$$\|v_{1}-v_{2}\|_{X_{T}}=\|\Gamma_{v_{0}}v_{1}-\Gamma_{v_{0}}v_{2}\|_{X_{T}}\lesssim\|v_{0}-v_{0}\|_{H^{s}(\R)}=0.$$

\textbf{Acknowledgments}: The authors gratefully acknowledge financial support by the Deu\-tsche Forschungs\-gemeinschaft (DFG) through CRC 1173. Wholehearted thanks also go to Vadim Zharnitsky for piquing our interest in the (missing) tooth problem for NLS. 

\end{section}

\end{document}